\documentclass[reqno,10pt]{amsart}
\usepackage{amscd,amssymb,verbatim,array}
\usepackage{hyperref}
\usepackage{amsmath}
\usepackage[active]{srcltx} % SRC Specials: DVI [Inverse] Search

\numberwithin{equation}{section} \theoremstyle{plain}

\newtheorem{theorem}{Theorem}[section]
\newtheorem{lemma}[theorem]{Lemma}
\newtheorem{proposition}[theorem]{Proposition}
\newtheorem{corollary}[theorem]{Corollary}
\newtheorem{definition}[theorem]{Definition}

\theoremstyle{definition}

\theoremstyle{remark}
\newtheorem{remark}[theorem]{Remark}

\numberwithin{equation}{section}

\newcommand{\LDet}{\operatorname{LDet}}
\newcommand{\Det}{\operatorname{Det}}

\newcommand{\even}{\operatorname{even}}

\newcommand{\odd}{\operatorname{odd}}

\newcommand{\Ker}{\operatorname{Ker}}

\newcommand{\im}{\operatorname{Im}}

\newcommand{\nablahsharp}{\nabla^{\mathcal{H},\sharp}}

\begin{document}

\title{Twisted Cappell-Miller holomorphic and analytic torsions}
\author{Rung-Tzung Huang}

\address{Institute of Mathematics, Academia Sinica,
6th Floor, Astronomy-Mathematics Building,
No. 1, Section 4, Roosevelt Road, Taipei, 106-17, Taiwan}

\email{rthuang@math.sinica.edu.tw}

\subjclass[2010]{Primary: 58J52 }
\keywords{determinant, analytic torsion}
%\date{\today}

\begin{abstract}
Recently, Cappell and Miller extended the classical construction of the analytic torsion for de Rham complexes to coupling with an arbitrary flat bundle and the holomorphic torsion for $\bar{\partial}$-complexes to coupling with an arbitrary holomorphic bundle with compatible connection of type $(1,1)$. Cappell and Miller studied the properties of these torsions, including the behavior under metric deformations. On the other hand, Mathai and Wu generalized the classical construction of the analytic torsion to the twisted de Rham complexes with an odd degree closed form as a flux and later, more generally, to the $\mathbb{Z}_2$-graded elliptic complexes. Mathai and Wu also studied the properties of analytic torsions for the $\mathbb{Z}_2$-graded elliptic complexes, including the behavior under metric and flux deformations. In this paper we define the Cappell-Miller holomorphic torsion for the twisted Dolbeault-type complexes and the Cappell-Miller analytic torsion for the twisted de Rham complexes. We obtain variation formulas for the twisted Cappell-Miller holomorphic and analytic torsions under metric and flux deformations.  
\end{abstract}
\maketitle

\section{Introduction}
In the celebrated works \cite{RS1,RS2}, Ray and Singer defined the analytic torsion for de Rham complexes and the holomorphic torsion for $\bar{\partial}$-complexes of complex manifolds. Ray and Singer studied the properties of analytic and holomorhpic torsions, including the behavior under metric deformations. In their works, Ray and Singer coupled the Riemannian Laplacian and the $\bar{\partial}$-Laplacian, respectively, with unitary flat vector bundles and yielded self-adjoint operators. Hence, the analytic torsion and holomorphic torsion are real numbers in the acyclic cases considered by Ray and Singer and are expressed as elements of real determinant line bundles. 

Recently, Cappell and Miller \cite{CM} extended the classical construction of the analytic torsion to coupling with an arbitrary flat bundle and the holomorphic torsion to coupling with an arbitrary holomorphic bundle with compatible connection of type $(1,1)$. This includes both unitary and flat (not necessarily unitary) bundles as special cases. However, in this general setting, the associated operators are not necessarily self-adjoint and the torsions are complex-valued. Cappell and Miller also studied the properties of their torsions, including the behavior under metric deformations.\textit{\textit{}}

In \cite{MW,MW2} Mathai and Wu generalized the classical construction of the Ray-Singer torsion for de Rham complexes to the twisted de Rham complex with an odd degree closed differential form $H$ as a flux and later, more generally, in \cite{MW3}, to the $\mathbb{Z}_2$-graded elliptic complexes. The definitions use pseudo-differential operators and residue traces. Mathai and Wu also studied propoerties of analtic torsion for $\mathbb{Z}_2$-graded elliptic complexes, including the behavior under the variation of metric and flux. 

Let $E$ be a holomorphic bundle with a compatible type $(1,1)$ connection $D$, cf. Definition \ref{D:compconn1}, over a complex manifold $W$ of complex dimension $n$ and $H \in A^{0,\bar{1}}(W,\mathbb{C})$ be a $\bar{\partial}$ closed differential form, for each $p, 1 \le p \le n$, we define the twisted Cappell-Miller holomorphic torsion, $\tau_{\operatorname{holo},p}(W,E,H)$, cf. Definition \ref{D:maindef1}, as a non-vanishing element of the determinant line
\[
\tau_{\operatorname{holo},p}(W,E,H) \in \Det H^{p,\bullet}_{\bar{\partial}_E}(W,E,H) \otimes [\Det H^{\bullet,n-p}_{D^{1,0}}(W,E,H)]^{(-1)^{n+1}}.
\]
We show that the variation of the twisted Cappell-Miller holomorphic torsion $\tau_{\operatorname{holo},p}(W,E,H)$ under the deformation of the metric is given by a local formula, cf. Theorem \ref{T:vartheorem1}. We also show that along any deformation of $H$ that fixes the cohomology class $[H]$ and the natural identification of determinant lines, the variation of the twisted Cappell-Miller holomorphic torsion $\tau_{\operatorname{holo},p}(W,E,H)$ under the deformation of the flux is given by a local formula, cf. Theorem \ref{T:vartheorem2}.

Let $\mathcal{E}$ be a complex flat vector bundle over a closed manifold $M$ endowed with a flat connection $\nabla$ and let $\mathcal{H}$ be an odd degree flux form.
Then the Cappell-Miller analytic torsion $\tau(\nabla,\mathcal{H})$, cf. Definition \ref{D:cmantorsi}, for the twisted de Rham complexes is an element of $\Det H^\bullet(M,\mathcal{E} \oplus \mathcal{E}',\mathcal{H})$. We show that the variation of the twisted Cappell-Miller analytic torsion $\tau(\nabla,\mathcal{H})$ under the deformation of the metric is given by a local formula, cf. Theorem \ref{T:anvartheorem1}. We also show that along any deformation of $\mathcal{H}$ that fixes the cohomology class $[\mathcal{H}]$ and the natural identification of determinant lines, the varition of the the twisted Cappell-Miller analytic torsion $\tau(\nabla,\mathcal{H})$ under the deformation of the flux is given by a local formula, cf. Theorem \ref{T:anvartheorem1}. In particular, we show that if the manifold $M$ is an odd dimensional closed oriented manifold, then the twisted Cappell-Miller analytic torsion is independent of the Riemannian metric and the representative $\mathcal{H}$ in the cohomology class $[\mathcal{H}]$. See also \cite[Section 6]{Su}. We also compare the twisted Cappell-Miller analytic torsion with the twisted refined analytic torsion \cite{H1}, cf. Theorem \ref{T:cmrefantorcomp}. 

Note that in \cite{H1} the author defined and studied the refined analytic torsion of Braverman and Kappeler \cite{BK1,BK3} for the twisted de Rham complexes. Later, in \cite{Su}, Su defined and studied the Burghelea-Haller analytic torsion \cite{BH1,BH2,BH3} for the twisted de Rham complexes and compared the twisted Burghelea-Haller torsion with the twisted refined analytic torsion. In \cite{Su}, Su also briefly discussed the twisted Cappell-Miller analytic torsion when the dimension of the manifold is odd.

The rest of the paper is organized as follows. In Section 2, we define and calculate the Cappell-Miller torsion for the $\mathbb{Z}_2$-graded finite dimensional bi-graded complex. In section 3, we first define the Dolbeault-type bi-graded complexes twisted by a flux form and its (co)homology groups. We then define the Cappell-Miller holomorphic torsion for the twisted Dolbeault-type bi-graded complexes. We prove variation theorems for the twisted Cappell-Miller holomorphic torsion under metric and flux deformations. In section 4, we first define the de Rham bi-graded complex twisted by a flux form and its (co)homology groups. Then we define the Cappell-Miller analytic torsion for the twisted de Rham bi-graded complex. We prove variation theorems for the twisted Cappell-Miller analytic torsion under metric and flux deformations. 

Throughout this paper, the bar over an integer means taking the value modulo $2$.

\section{The Cappell-Miller torsion for a $\mathbb{Z}_2$-graded finite dimensional bi-graded complex} 
In this section we define and calculate the Cappell-Miller torsion for the $\mathbb{Z}_2$-graded finite dimensional bi-graded complex. For the $\mathbb{Z}$-graded case, cf. \cite[Section 6]{CM}. Throughout this section ${\bf k}$ is a field of characteristic zero.

\subsection{The determinant lines of a $\mathbb{Z}_2$-graded finite dimensional bi-graded complex}
Given a $\bf{k}$-vector space $V$ of dimension $n$, the determinant line of $V$ is the line $\Det(V):=\wedge^nV$, where $\wedge^nV$ denotes the $n$-th exterior power of $V$. By definition, we set $\operatorname{Det}(0):=\bf{k}$. Further, we denote by $\Det(V)^{-1}$ the dual line of $\Det(V)$.  
Denote by 
$$C^{\bar{0}}=C^{\operatorname{even}}=\bigoplus_{i=0}^{[\frac{m}{2}]}C^{2i} , \qquad C^{\bar{1}}=C^{\operatorname{odd}}=\bigoplus_{i=0}^{[\frac{m-1}{2}]}C^{2i+1},$$
where $C^i,i=0,\cdots,m$, are finite dimensional $\bf{k}$-vector spaces. 
Let 
\begin{equation}\label{E:detline} 
 (C^\bullet,d) \  : \  \cdots \stackrel{d}{\longrightarrow} \ C^{\bar{0}} \ \stackrel{d}{\longrightarrow}\ C^{\bar{1}}\ \stackrel{d}{\longrightarrow} \ C^{\bar{0}} \ \stackrel{d}{\longrightarrow} \cdots
\end{equation}
be a $\mathbb{Z}_2$-graded cochain complex of finite dimensional $\bf{k}$-vector spaces. Denote by $H^{\bullet}(d)=H^{\bar 0}(d)\oplus H^{\bar 1}(d)$ its cohomology. Set
\begin{equation}
\Det(C^\bullet)\,:=\,\Det \big( C^{\bar{0}} \big) \otimes \Det \big( C^{\bar{1}} \big)^{-1}, \qquad \Det(H^\bullet(d))\,:=\, \Det( H^{\bar{0}}(d)) \otimes \Det( H^{\bar{1}}(d))^{-1}.
\end{equation}

Now if in addition, $C^\bullet$ has another differential $d^*:C^{\bar{k}} \to C^{\overline{k-1}}$ so that
\[
 (C^\bullet,d^*) \  : \  \cdots \stackrel{d^*}{\longleftarrow} \ C^{\bar{0}} \ \stackrel{d^*}{\longleftarrow}\ C^{\bar{1}}\ \stackrel{d^*}{\longleftarrow} \ C^{\bar{0}} \ \stackrel{d^*}{\longleftarrow} \cdots. 
\]
Denote by $H_\bullet(d^*)=H_{\bar{0}}(d^*) \oplus H_{\bar{1}}(d^*)$ its homology. Set
\[
\Det(H_\bullet(d^*))\,:=\, \Det( H_{\bar{0}}(d^*)) \otimes \Det( H_{\bar{1}}(d^*))^{-1},
\]

\subsection{The fusion isomorphisms}(cf. \cite[Subsection 2.3]{BK3})
For two finite dimensional $\bf{k}$-vector spaces $V$ and $W$, we denote by $\mu_{V,W}$ the canonical fusion isomorphism,
\begin{equation}\label{E:fusio}
\mu_{V,W} : \Det(V) \otimes \Det(W) \to \Det(V \oplus W). 
\end{equation}
For $v \in \Det(V)$, $w \in \Det(W)$, we have
\begin{equation}\label{E:fusio1}
\mu_{V,W}(v \otimes w)= (-1)^{\dim V \cdot \dim W} \mu_{W,V}(w \otimes v).
\end{equation}
By a slight abuse of notation, denote by $\mu_{V,W}^{-1}$ the transpose of the inverse of $\mu_{V,W}$.

Similarly, if $V_1, \cdots, V_r$ are finite dimensional $\bf{k}$-vector spaces, we define an isomorphism
\begin{equation}\label{E:mulfus}
\mu_{V_1,\cdots, V_r}: \Det(V_1) \otimes \cdots \otimes \Det(V_r) \to \Det(V_1 \oplus \cdots \oplus V_r).
\end{equation}

\subsection{The isomorphism between determinant lines}\label{SS:isomo}
%%%%%%%%%%%%%%%%%%%%%%%%%%%%%%%%%%%%%%%%%%%55
For $k=0,1$, fix a direct sum decomposition
\begin{equation}\label{E:decompbha}
C^{\bar{k}}=B^{\bar{k}} \oplus H^{\bar{k}} \oplus A^{\bar{k}},
\end{equation}
such that $B^{\bar{k}}\oplus H^{\bar{k}}=(\Ker d) \cap C^{\bar{k}}$ and $B^{\bar{k}}= d\big(C^{\overline{k-1}}\big)=d\big(A^{\overline{k-1}}\big)$. Then $H^{\bar{k}}$ is naturally isomorphic to the cohomology $H^{\bar{k}}(d)$ and $d$ defines an isomorphism $d: A^{\bar{k}} \to B^{\overline{k+1}}$.

Fix $c_{\bar{k}} \in \Det(C^{\bar{k}})$ and $x_{\bar{k}} \in \Det(A^{\bar{k}})$. Let $d(x_{\bar{k}}) \in \Det(B^{\overline{k+1}})$ denote the image of $x_{\bar{k}}$ under the map $\Det(A^{\bar{k}}) \to \Det(B^{\overline{k+1}})$ induced by the isomorphism $d:A^{\bar{k}} \to B^{\overline{k+1}}$. Then there is a unique element $h_{\bar{k}} \in \Det(H^{\bar{k}})$ such that
\begin{equation}\label{E:cbahce}
c_{\bar{k}}= \mu_{B^{\bar{k}},H^{\bar{k}},A^{\bar{k}}}\big(d(x_{\overline{k-1}})\otimes h_{\bar{k}} \otimes x_{\bar{k}}\big),
\end{equation}
where $\mu_{B^{\bar{k}},H^{\bar{k}},A^{\bar{k}}}$ is the fusion isomorphism, cf. \eqref{E:mulfus}, see also \cite[Subsection 2.3]{BK3}.

Define the canonical isomorphism
\begin{equation}\label{E:isomorphism}
\phi_{C^\bullet}=\phi_{(C^{\bullet},d)} \,:\,\Det(C^{\bullet} ) \longrightarrow \Det(H^{\bullet}(d)),
\end{equation}
by the formula
\begin{equation}\label{E:phic}
\phi_{C^\bullet}: c_{\bar{0}} \otimes c_{\bar{1}}^{-1} \mapsto  h_{\bar{0}} \otimes h_{\bar{1}}^{-1}. 
\end{equation}

Notice that, following the sign convention of \cite[(2-14)]{BK3}, in \cite[(2.10)]{H1} a sign refined version of the canonical isomorphism \eqref{E:isomorphism} was introduced. Here we follow the sign convention of \cite[Section 6]{CM}.
%%%%%%%%%%%%%%%%%%%%%%%%%%%%%%%%%%%%%%%%%%%%%%%%%%%%%%%%%%%%%%%%%%%%%%%%

Similarly, for $k=0,1$, fix a direct sum decomposition
\begin{equation}\label{E:decompbhahom}
C^{\bar{k}}=B_{\bar{k}} \oplus H_{\bar{k}} \oplus A_{\bar{k}},
\end{equation}
such that $B_{\bar{k}}\oplus H_{\bar{k}}=(\Ker d^*) \cap C^{\bar{k}}$ and $B_{\bar{k}}= d^*\big(C^{\overline{k+1}}\big)=d^*\big(A_{\overline{k+1}}\big)$. Then $H_{\bar{k}}$ is naturally isomorphic to the homology $H_{\bar{k}}(d^*)$ and $d^*$ defines an isomorphism $d^*: A_{\bar{k}} \to B_{\overline{k-1}}$.

Similarly, fix $c_{\bar{k}} \in \Det(C^{\bar{k}})$ and $y_{\bar{k}} \in \Det(A_{\bar{k}})$. Let $d^*(y_{\bar{k}}) \in \Det(B_{\overline{k-1}})$ denote the image of $y_{\bar{k}}$ under the map $\Det(A_{\bar{k}}) \to \Det(B_{\overline{k-1}})$ induced by the isomorphism $d^*:A_{\bar{k}} \to B_{\overline{k-1}}$. Then there is a unique element $h'_{\bar{k}} \in \Det(H^{\bar{k}})$ such that
\begin{equation}\label{E:cbahceho}
c_{\bar{k}}= \mu_{B_{\bar{k}},H_{\bar{k}},A_{\bar{k}}}\big(d^*(y_{\overline{k+1}})\otimes h'_{\bar{k}} \otimes y_{\bar{k}}\big),
\end{equation}
where $\mu_{B_{\bar{k}},H_{\bar{k}},A_{\bar{k}}}$ is the fusion isomorphism, cf. \eqref{E:mulfus}, see also \cite[Subsection 2.3]{BK3}.

Define the canonical isomorphism
\begin{equation}\label{E:isomorphismho}
\phi'_{C^\bullet}=\phi'_{(C^{\bullet},d^*)} \,:\,\Det(C^{\bullet} ) \longrightarrow \Det(H_{\bullet}(d^*)),
\end{equation}
by the formula
\begin{equation}\label{E:phicho}
\phi'_{C^\bullet}: c_{\bar{0}} \otimes c_{\bar{1}}^{-1} \mapsto  h'_{\bar{0}} \otimes {h'}_{\bar{1}}^{-1}. 
\end{equation}

%%%%%%%%%%%%%%%%%%%%%%%%%%%%%%%%%%%%%%%%%%%%%%%%%%%%%%%%%%%%%%%%%%%%%%%%%%%%%%%%%%%%%%%%%%%
\subsection{The Cappell-Miller torsion for a $\mathbb{Z}_2$-graded finite dimensional bi-graded complex}
Let $C^\bullet=C^{\bar 0} \oplus C^{\bar 1}$ and $\widetilde{C}^\bullet=\widetilde{C}^{\bar 0} \oplus \widetilde{C}^{\bar 1}$ be finite dimensional $\mathbb{Z}_2$-graded $\bf{k}$-vector spaces. The fusion isomorphism 
\[
\mu_{C^\bullet ,\widetilde{C}^{\bullet}} : \Det(C^{\bullet}) \otimes \Det(\widetilde{C}^{\bullet}) \to \Det(C^{\bullet} \oplus \widetilde{C}^{\bullet}),
\]
is defined by the formula
\begin{equation}\label{E:grafus}
\mu_{C^{\bullet},\widetilde{C}^{\bullet}} := (-1)^{\mathcal{M}(C^{\bullet},\widetilde{C}^{\bullet})}  \mu_{C^{\bar 0},\widetilde{C}^{\bar 0}} \otimes \mu_{C^{\bar 1},\widetilde{C}^{\bar 1}}^{-1},
\end{equation}
where 
\begin{equation}\label{E:mvw123}
\mathcal{M}(C^{ \bullet},\widetilde{C}^{ \bullet}) := \dim C^{\bar 1} \cdot \dim \widetilde{C}^{\bar 0}.
\end{equation}

Consider the element
\[
c\, := \, c_{\bar{0}} \otimes c_{\bar{1}}^{-1} 
\]
of $\Det(C^\bullet)$. Then, for the bi-graded complex $(C^\bullet,d,d^*)$, the Cappell-Miller torsion is the algebraic torsion invariant
\begin{equation}\label{E:tauddstar}
\tau(C^\bullet,d,d^*) \, := \, (-1)^{S(C^\bullet)} \phi_{C^\bullet}(c)(\phi'_{C^\bullet}(c))^{-1} \,  \in \, \Det(H^\bullet(d)) \otimes \Det(H_\bullet(d^*))^{-1},
\end{equation} 
where $(-1)^{S(C^\bullet)}$ is defined by the formula
\begin{equation}\label{E:scbu}
S(C^\bullet):= \sum_{k=0,1}\big[ \dim B_{\overline{k-1}}\cdot \dim B^{\overline{k+1}} + \dim B^{\overline{k+1}}\cdot \dim H_{\bar k} + \dim B_{\overline{k-1}} \cdot \dim H^{\bar k}\big].
\end{equation}

\subsection{Calculation of the $\mathbb{Z}_2$-graded Cappell-Miller torsion}

In this subsection we compute the $\mathbb{Z}_2$-graded Cappell-Miller torsion. We first compute the case that the combinatorial Laplacian $\Delta:=d^*d+dd^*$ is bijective.

For $k=0,1$, define
\begin{equation}\label{E:cplmidec}
C_+^{\bar k}:= \Ker d^* \cap C^{\bar k}, \qquad C_-^{\bar k}:= \Ker d \cap C^{\bar k}.
\end{equation}
%Let $\Delta_{\bar k}$ denote the restriction of $\Delta$ to $C^{\bar k}$.
The proof of the following proposition is similar to the proof of the $\mathbb{Z}$-graded case, \cite[Subsection 6.2, Claim B]{CM}.

\begin{proposition}\label{P:acydob1}
Suppose that the combinatorial Laplacian $\Delta$ has no zero eigenvalue. Then the cohomology group $H^{\bullet}(d)=0$ and the homology group $H_{\bullet}(d^*)=0$. Moreover, 
\begin{equation}\label{E:acytauddsta}
\tau(C^\bullet,d,d^*)=\Det(d^*d|_{C_+^{\bar 0}}) \cdot \Det(d^*d|_{C_+^{\bar 1}})^{-1},
\end{equation} 
\end{proposition}

\begin{proof}
If the combinatorial Laplacian $\Delta$ has no zero eigenvalue, then, for each $k=0,1$, $\Delta$ is an isomorphism from $C^{\bar k}$ to $C^{\bar k}$. Hence, for each $x \in C^{\bar k}$, there is a $y \in C^{\bar k}$ such that $$x = \Delta y = (dd^*+d^*d)y.$$ This implies that $$C^{\bar k}=d^*C^{\overline{k+1}}+dC^{\overline{k-1}}.$$ 

If $x \in d^* C^{\overline{k+1}} \cap d C^{\overline{k-1}}$, then, by the facts that $d^* d^*=0$ and $dd=0$, we have $d^*x=0$ and $dx=0$. Hence, we have $$\Delta x = (dd^*+d^*d)x=0,$$ which implies that $x=0$ by the assumption. We conclude that, for each $k=0,1$, we have the following direct sum decomposition
\begin{equation}\label{E:splitck}
C^{\bar k}=d^*C^{\overline{k+1}} \oplus dC^{\overline{k-1}}.
\end{equation} 

By the facts that 
\[
d^*\Delta=d^*dd^*=\Delta d^*, \qquad d\Delta=dd^*d=\Delta d,
\]
we know that the isomorphism $\Delta$ preserves the splitting \eqref{E:splitck}. Hence $\Delta$ maps $d^*C^{\overline{k+1}}$ and $dC^{\overline{k-1}}$ isomorphically to themselves. Since 
\[
\Delta|_{d^*C^{\overline{k+1}}}=d^*d|_{d^*C^{\overline{k+1}}}, \qquad \Delta|_{d_{\overline{k}}C^{\overline{k}}}=dd^*|_{d_{\overline{k}}C^{\overline{k}}}, 
\]
the maps $d^*d$ and $dd^*$ are isomorphisms on $d^*C^{\overline{k+1}}$ and $dC^{\overline{k}}$, respectively. In particular, $d:d^*C^{\overline{k+1}} \to dC^{\bar k}$ and $d^*:dC^{\bar k} \to d^*C^{\overline{k+1}}$ are injections, respectively, and the composite $d^*d$ is an isomorphism from $d^*C^{\overline{k+1}}$ to itself. Hence we have the following isomorphisms 
\begin{equation}\label{E:dkdsatrk0}
d: d^*C^{\overline{k+1}} \cong dC^{\bar k}, \qquad  d^*: dC^{\bar k} \cong d^*C^{\overline{k+1}} 
\end{equation}
and, in particular, $H_{\bullet}(d^*)=0$ and $H^{\bullet}(d)=0$. This proves the first assertion.

To compute $\tau(C^\bullet,d,d^*)$, cf. \eqref{E:tauddstar}, we first compute $\phi'_{C^\bullet}(c)$. 
By \eqref{E:cplmidec}, \eqref{E:splitck} and \eqref{E:dkdsatrk0},
we know that 
\begin{equation}\label{E:cpluminre}
C_+^{\bar k}=d^*C^{\overline{k+1}}, \qquad   C_-^{\bar k}=dC^{\overline{k-1}}.
\end{equation}
By \eqref{E:decompbha}, \eqref{E:decompbhahom}, \eqref{E:dkdsatrk0}, \eqref{E:cpluminre} and the first assertion, we know that 
\begin{equation}\label{E:baisomorphism}
C_+^{\bar k}=B_{\bar k} \cong A^{\overline{k}}, \qquad C_-^{\bar k}=B^{\bar k} \cong A_{\overline{k}}. 
\end{equation}
Let $\{ d^*y_{\overline{k+1},i} | 1 \le i \le \dim B_{\bar k}\}$ be a basis for $B_{\bar k}=d^*C^{\overline{k+1}} \cong A^{\bar k}.$ Since 
\[
d^*d: d^*C^{\overline{k+1}} \to d^*C^{\overline{k+1}}
\]
is an isomorphism, there is a unique vector 
\[
x_{\bar{k},i} \in B_{\bar k}=d^*C^{\overline{k+1}}
\] 
such that 
\begin{equation}\label{E:dstaktx}
d^*dx_{\bar{k},i}=d^*y_{\overline{k+1},i}.
\end{equation}
Then $\{x_{\overline{k},i} | 1 \le i \le \dim B_{\bar k}\}$ is also a basis for $B_{\bar k}\cong A^{\bar k}$. By the fact that $d:d^*C^{\overline{k+1}} \to dC^{\bar k}$ is an isomorphism, it follows that $\{dx_{\overline{k},i} | 1 \le i \le \dim B_{\bar k}\}$ is a basis for $B^{\overline{k+1}}=dC^{\bar k} \cong A_{\overline{k+1}}$.
Hence, in view of the decomposition \eqref{E:splitck}, we conclude that 
\[
\{ d^*y_{\overline{k+1},i} | 1 \le i \le \dim B_{\bar k}\} \cup \{dx_{\overline{k-1},i} | 1 \le i \le \dim B_{\overline{k-1}}\}
\]
forms a basis for $C^{\bar k}$. In particular, by the first assertion and \eqref{E:decompbha}, we have
\begin{equation}\label{E:dimequabb}
\dim B^{\bar k} = \dim B_{\overline{k-1}}.
\end{equation}
With this particular choice of basis, we set 
\[
y_{\overline{k+1}}:= y_{\overline{k+1},1} \wedge \cdots \wedge y_{\overline{k+1},\dim B_{\bar k}} \in \Det(A_{\overline{k+1}})
\]
and 
\[
x_{\overline{k-1}}:= x_{\overline{k-1},1} \wedge \cdots \wedge x_{\overline{k-1},\dim B_{\overline{k-1}}} \in \Det(B_{\overline{k-1}}).
\]
Let $d^*y_{\overline{k+1}}$ be the induced element in $\Det(B_{\bar k})$ and $dx_{\overline{k-1}}$ be the induced element in $\Det(A_{\bar k})$. Set 
\begin{equation}\label{E:cbarkcho}
c_{\bar k}=\mu_{B_{\bar k}, A_{\bar k}}(d^*y_{\overline{k+1}} \otimes dx_{\overline{k-1}}).
\end{equation}
To compute $\phi'_{C^\bullet}(c)$, cf. \eqref{E:phicho}, we need to compute the element $h'_{\bar{k}} \in \Det(H_{\bar{k}}(d^*) \cong \bf{k}$. 

If $L$ is a complex line and $x,y \in L$ with $y \not=0$, we denote by $[x:y] \in \bf{k}$ the unique number such that $x=[x:y]y$. Then
\begin{align}\label{E:wihheo}
h'_{\bar k}& = [c_{\bar k}: \mu_{B_{\bar k}, A_{\bar k}}(d^*y_{\overline{k+1}} \otimes dx_{\overline{k-1}})] \quad \text{by} \ \eqref{E:dstaktx} \nonumber \\
 & =  [\mu_{B_{\bar k}, A_{\bar k}}(d^*y_{\overline{k+1}} \otimes dx_{\overline{k-1}}): \mu_{B_{\bar k}, A_{\bar k}}(d^*y_{\overline{k+1}} \otimes dx_{\overline{k-1}})] \quad \text{by} \ \eqref{E:cbarkcho} \\
 & =  1. \nonumber
\end{align}

We next compute $\phi_{C^\bullet}(c)$. By \eqref{E:phic}, we need to compute $h_{\bar k}$. By the above choice of basis, we have
\begin{align}\label{E:wihheo1}
h_{\bar k} & =  [c_{\bar k}: \mu_{B^{\bar k}, A^{\bar k}}(dx_{\overline{k-1}}) \otimes x_{\bar k}] \nonumber \\
 & =  [\mu_{B_{\bar k}, A_{\bar k}}(d^*y_{\overline{k+1}} \otimes dx_{\overline{k-1}}): \mu_{B^{\bar k}, A^{\bar k}}(dx_{\overline{k-1}}) \otimes x_{\bar k}] \quad \text{by} \ \eqref{E:cbarkcho} \nonumber \\
& =  [\mu_{B_{\bar k}, A_{\bar k}}(d^*dx_{\bar k} \otimes dx_{\overline{k-1}}): \mu_{A_{\bar k}, B_{\bar k}}(dx_{\overline{k-1}}) \otimes x_{\bar k}] \quad \text{by} \ \eqref{E:baisomorphism}, \eqref{E:dstaktx} \\
& =  (-1)^{\dim B_{\bar k} \dim A_{\overline{k}}}[\mu_{B_{\bar k}, A_{\bar k}}(d^*dx_{\bar k} \otimes dx_{\overline{k-1}}): \mu_{B_{\bar k}, A_{\bar k}}( x_{\bar k} \otimes dx_{\overline{k-1}})] \nonumber \\
& =  (-1)^{\dim B_{\overline{k-1}} \dim B^{\overline{k+1}}} \Det(d^*d|_{C_+^{\bar k}}), \quad \text{by} \ \eqref{E:baisomorphism}, \eqref{E:dimequabb}. \nonumber 
\end{align}
By combining \eqref{E:tauddstar}, \eqref{E:scbu}, \eqref{E:wihheo}, \eqref{E:wihheo1} with the first assertion, we obtain \eqref{E:acytauddsta}.
\end{proof}

We now compute the case that the combinatorial Laplacian $\Delta:=d^*d+dd^*$ is {\em not} bijective.
Note that the operator $\Delta$ maps $C^{\bar{k}}$ into itself. For an arbitrary interval $\mathcal{I}$, denote by $C^{\bar{k}}_{\mathcal{I}} \subset C^{\bar{k}}$ the linear span of the generalized eigenvectors of the restriction of $\Delta$ to $C^{\bar{k}}$, corresponding to eigenvalue $\lambda$ with $\lambda \in \mathcal{I}$. Since both $d$ and $d^*$ commute with $\Delta$, $d (C^{\bar{k}}_{\mathcal{I}}) \subset C^{\overline{k+1}}_{\mathcal{I}}$ and $d^*(C^{\bar{k}}_{\mathcal{I}}) \subset C^{\overline{k-1}}_{\mathcal{I}}$. Hence, we obtain a subcomplex $C^{\bullet}_{\mathcal{I}}$ of $C^{\bullet}$. We denote by $H^{\bullet}_{\mathcal{I}}(d)$ the cohomology of the complex $(C^\bullet_{\mathcal{I}},d_{\mathcal{I}})$ and $H_{\bullet,\mathcal{I}}(d^*)$ the homology of the complex $(C^\bullet_{\mathcal{I}},d^*_{\mathcal{I}})$. Denote by $d_{\mathcal{I}}$  and $d^*_{\mathcal{I}}$ the restrictions of $d$ and $d^*$ to $C^{\bar{k}}_{\mathcal{I}}$ and denote by $\Delta_{\mathcal{I}}$ the restriction of $\Delta$ to $C^{\bar{k}}_{\mathcal{I}}$. Then $\Delta_{\mathcal{I}} = d^*_{\mathcal{I}} d_{\mathcal{I}}+d_{\mathcal{I}}d^*_{\mathcal{I}}$. For $k=0,1$, we also denote by $C_{\pm,\mathcal{I}}^{\bar k}$ the restrictions of $C_{\pm}^{\bar k}$ to $C_{\mathcal{I}}^{\bar k}$.

%\begin{lemma}
%If $0 \notin \mathcal{I}$, then the bi-graded complex $(C^{\bullet}_{\mathcal{I}},d_{\mathcal{I}},d^*_{\mathcal{I}})$ is doubly acyclic.
%\end{lemma}
%\begin{proof}
%If, for $k=0,1$, $x \in \Ker d_{\bar{k},\mathcal{I}}$, then $\Delta_{\bar{k},\mathcal{I}}x = d_{\overline{k+1}}d^*_{\overline{k}}x \in \im d_{\overline{k+1},\mathcal{I}} \subset \Ker d_{\overline{k},\mathcal{I}}$. Since the operators $\Delta_{\bar{k},\mathcal{I}}:C^{\bar{k}}_{\mathcal{I}} \to C^{\bar{k}}_{\mathcal{I}},k=0,1$ are invertible, we conclude that $\Ker d_{\overline{k},\mathcal{I}} = \im d_{\overline{k+1},\mathcal{I}}$. Similarly, we have $\Ker d^*_{\overline{k},\mathcal{I}} = \im d^*_{\overline{k+1},\mathcal{I}}$.
%\end{proof}

For each $\lambda \ge 0$, we have $C^{\bullet}=C^{\bullet}_{[0,\lambda]} \oplus C^{\bullet}_{(\lambda,\infty)}$. Then $H^{\bullet}(d)=0$ whereas $H^{\bullet}_{[0,\lambda]}(d) \cong H^{\bullet}(d)$ and $H_{\bullet}(d^*)=0$ whereas $H_{\bullet,[0,\lambda]}(d^*) \cong H_{\bullet}(d^*)$. Hence there are canonical isomorphisms
\[
\Phi_\lambda : \Det(H^{\bullet}_{(\lambda,\infty)}(d)) \to \mathbb{C}, \qquad \Psi_{\lambda}: \Det(H^{\bullet}_{[0,\lambda]}(d)) \to \Det(H^{\bullet}(d)).
\]
and
\[
\Phi'_\lambda : \Det(H_{\bullet,(\lambda,\infty)}(d^*)) \to \mathbb{C}, \qquad {\Psi'}^*_{\lambda}: \Det(H_{\bullet,[0,\lambda]}(d^*))^{-1} \to \Det(H_{\bullet}(d^*))^{-1}.
\]
In the sequel, we will write $t$ for $\Phi_\lambda(t) \in \mathbb{C}$ and $t'$ for $\Phi'_\lambda(t') \in \mathbb{C}$. 

\begin{proposition}\label{P:indedec}
Let $(C^{\bullet},d,d^*)$ be a $\mathbb{Z}_2$-graded bi-graded complex of finite dimensional $\bf{k}$-vector spaces. Then, for each $\lambda \ge 0$,
\begin{equation}\label{E:taucddstarinde}
\tau(C^\bullet,d,d^*)= \Det \big( d^*d|_{C_{+,(\lambda,\infty)}^{\bar 0}}\big) \cdot \Det\big(d^*d|_{C_{+,(\lambda,\infty)}^{\bar 1}}\big)^{-1} \cdot \tau(C^\bullet_{[0,\lambda]},d,d^*),
\end{equation}
where we view $\tau(C^\bullet_{[0,\lambda]},d,d^*)$ as an element of $\Det(H^{\bullet}(d))\otimes \Det(H_{\bullet}(d^*))^{-1}$ via the canonical isomorphism $\Psi_{\lambda} \otimes {\Psi'}^*_{\lambda}: \Det(H^{\bullet}_{[0,\lambda]}(d))\otimes \Det(H_{\bullet,[0,\lambda]}(d^*))^{-1} \to \Det(H^{\bullet}(d))\otimes \Det(H_{\bullet}(d^*))^{-1}$.

In particular, the right side of \eqref{E:taucddstarinde} is indenpendent of $\lambda \ge 0$.
\end{proposition}
\begin{proof}
Recall the natural isomorphism 
\begin{equation}\label{E:isofu}
\Det(H^{\bar{k}}_{[0,\lambda]}(d) \otimes H^{\bar{k}}_{(\lambda,\infty)}(d)) \cong \Det(H^{\bar{k}}_{[0,\lambda]}(d) \oplus H^{\bar{k}}_{(\lambda,\infty)}(d))=\Det(H^{\bar{k}}(d))
\end{equation}
and
\begin{equation}\label{E:isofu1}
\Det(H_{\bar{k},[0,\lambda]}(d^*) \otimes H_{\bar{k},(\lambda,\infty)}(d^*)) \cong \Det(H_{\bar{k},[0,\lambda]}(d^*) \oplus H_{\bar{k},(\lambda,\infty)}(d^*))=\Det(H_{\bar{k}}(d^*))
\end{equation}
From~\eqref{E:tauddstar}, Proposition~\ref{P:acydob1}, \eqref{E:isofu} and \eqref{E:isofu1} we obtain the result.
\end{proof}

\section{Twisted Cappell-Miller holomorphic torsion}

In this section we first review the $\bar{\partial}$-Laplacian for a holomorphic bundle with compatible type (1,1) connection introduced in \cite{CM}. Then we define the Dolbeault-type bi-graded complexes twisted by a flux form and its cohomology and homology groups. We define the Cappell-Miller holomorphic torsion for the twisted Dolbeault-type bi-graded complexes. We also prove variation theorems for the twisted Cappell-Miller holomorphic torsion under metric and flux deformations. 

\subsection{The $\bar{\partial}$-Laplacian for a holomorphic bundle with compatible type (1,1) connection}
In this subsection we review some materials from \cite{CM}, see also \cite{LY}.

Let $(W,J)$ be a complex manifold of complex dimension $n$ with the complex structure $J$ and let $g^{W}$ be any Hermitian metric on $TW$. Let $E \to W$ be a holomorphic bundle over $W$ endowed with a linear connection $D$ and let $h^E$ be a Hermitian metric on $E$. 

The complex structure $J$ induces a splitting $TW\otimes_{\mathbb{R}}\mathbb{C}=T^{(1,0)}W \oplus T^{(0,1)}W$, where $T^{(1,0)}W$ and $T^{(0,1)}W$ are eigenbundles of $J$ corresponding to eigenvalues $i$ and $-i$, respectively. Let $T^{*(1,0)}W$ and $T^{*(0,1)}W$ be the corresponding dual bundles. For $0 \le p,q \le n$, let
\[
A^{p,q}(W,E)=\Gamma \big( W, \wedge^p(T^{*(1,0)}W)\otimes\wedge^q(T^{*(0,1)}W)\otimes E \big)
\]
be the space of smooth $(p,q)$-forms on $W$ with values in $E$. Set 
\[
A^{\bullet,\bullet}(W,E)=\bigoplus_{p,q=0}^n A^{p,q}(W,E).
\]
%Denote by $\Omega^r(W,E), 0 \le r \le 2n$ the space of smooth $r$-forms on $W$ with values in $E$. We have the following direct sum decomposition
%\[
%\Omega^r(W,E)=
%\]
Let $\bar{\partial}: A^{p,q}(W,\mathbb{C}) \to W^{p,q+1}(W,\mathbb{C})$ and $\partial: A^{p,q}(W,\mathbb{C}) \to A^{p+1,q}(W,\mathbb{C})$ be the standard operators obtained by decomposing, by type, the exterior derivative 
\[
d=\bar{\partial}+\partial
\]
acting on complex-valued smooth forms of type $(p,q)$. By $d^2=0$, we have $\bar{\partial}^2=0, \, \partial^2=0$. 

Since $E$ is holomorphic, the operator on $A^{\bullet,\bullet}(W,E)$ has a unique natural extension to $A^{\bullet,\bullet}(W,E)$, cf. \cite[P.139]{CM},
\[
\bar{\partial}_E : A^{p,q}(W,E) \to W^{p,q+1}(W,E).
\]

Under the splitting $\Gamma(W, (T^*W \otimes_{\mathbb{R}}\mathbb{C})\otimes_{\mathbb{C}}E)=A^{1,0}(W,E) \oplus A^{0,1}(W,E)$, the connection $D$ decomposes as a sum, $D=D^{1,0} \oplus D^{0,1}$ with
\[
D^{1,0}:\Gamma(W,E) \to A^{1,0}(W,E), \quad D^{0,1}: \Gamma(W,E) \to A^{0,1}(W,E).
\]
Extend the connection $D$ on $\Gamma(W,E)$ in a unique way to $A^{\bullet,\bullet}(W,E)$ by the Leibniz formula, cf. \cite[P. 21]{BGV}, then the extended $D$ again decomposes as a sum $D = D^{1,0} + D^{0,1}$, which also satisfy the Leibniz formula, cf. \cite[P. 131]{BGV}.   

We recall the following definition, cf. \cite[P.139-140]{CM} or \cite[Definition 2.1]{LY}.
\begin{definition}\label{D:compconn1}
The connection $D$ is said to be {\em compatible} with the holomorphic structure on $E$ if $D^{0,1}=\bar{\partial}_E$. The coonection $D$ is said to be of {\em type} $(1,1)$ if the curvature $D^2$ is of type $(1,1)$, that is, $(D^{1,0})^2=0$ and $(D^{0,1})^2=0$.
\end{definition}

The complex Hodge star operator $\star$ acting on forms is a complex conjugate linear mapping
\[
\star:A^{p,q}(W,\mathbb{C}) \to A^{n-p,n-q}(W,\mathbb{C})
\]
induced by a conjugate linear bundle isomorphism, cf. \cite[P. 141]{CM}. 

The natural conjugate mapping 
\[
\operatorname{conj}: A^{p,q}(W,\mathbb{C}) \to A^{q,p}(W,\mathbb{C})
\]
is a complex linear mapping, induced by the bundle automorphism, cf. \cite[P. 141]{CM},
\[
T^*W\otimes_{\mathbb{R}}\mathbb{C} \to T^*W\otimes_{\mathbb{R}}\mathbb{C}, \ v \otimes \lambda \mapsto v \otimes \bar{\lambda}, \quad \ v \in T^*W, \, \lambda \in \mathbb{C},
\]
of the complexified cotangent bundle. Denote by $\hat{\star}:=\operatorname{conj}\star$. Then
\[
\hat{\star}=\operatorname{conj}\star : A^{p,q}(W,\mathbb{C}) \to A^{n-q,n-p}(W,\mathbb{C})
\]
is a complex linear mapping. Clearly, $\hat{\star}=\operatorname{conj}\star=\star\operatorname{conj}$.

As pointed out in \cite[P. 141]{CM} that $\hat{\star}$ being complex linear may be coupled to a complex linear bundle mapping, such as the identity mapping. We also denote by $\hat{\star}$ the complex linear mapping
\[
\hat{\star} : A^{p,q}(W,E) \to A^{n-q,n-p}(W,E).
\]

Recall that the adjoint $\bar{\partial}^*$ of $\bar{\partial}$ with respect to the chosen Hermitian inner product on $TW$ is given by
\[
\bar{\partial}^* = - \star \bar{\partial} \star.
\]
In particular,
\[
\bar{\partial}^*=-\hat{\star}\operatorname{conj} \bar{\partial} \operatorname{conj}\hat{\star}=-\hat{\star} \partial \hat{\star}.
\]

Let $D$ be a compatible $(1,1)$ connection. Following \cite[P. 141]{CM}, we define 
\[
\bar{\partial}^*_{E,D^{1,0}}=-\hat{\star} D^{1,0} \hat{\star} 
\]
and the $\bar{\partial}$-Laplacian for the holomorphic bundle $E$ with compatible type $(1,1)$ connection $D$ by
\[
\square_{E,\bar{\partial}} = \bar{\partial}_E \bar{\partial}^*_{E,D^{1,0}} + \bar{\partial}^*_{E,D^{1,0}}\bar{\partial}_E.
\]
Note that $(\bar{\partial}^*_{E,D^{1,0}})^2=0$, since $(D^{1,0})^2=0$ and $\hat{\star}^2=\star^2=\pm 1$.

Denote by $\delta_E$ the adjoint of the $\bar{\partial}$-operator $\bar{\partial}_E$ with respect to the inner product $<\cdot,\cdot>_E$ on $A^{\bullet,\bullet}(W,E)$ induced by the Hermitian metrics $g^W$ and $h^E$. Then the associated self-adjoint $\bar{\partial}$-Laplacian is defined as
\[
\square_E = (\bar{\partial}_E + \delta_E)^2 = \bar{\partial}_E \delta_E + \delta_E \bar{\partial}_E.
\]

Recall that, in general, the operator $\square_{E,\bar{\partial}}$ is not self-adjoint with respect to the inner product $<\cdot,\cdot>_E$ on $A^{\bullet,\bullet}(W,E)$, but has the same leading symbol as the operator $\square_E$, cf. \cite[Section 3]{CM}. When the connection on $E$ is compatible with the Hermitian inner product $<\cdot,\cdot>_E$ on $A^{\bullet,\bullet}(W,E)$, the operator $\square_{E,\bar{\partial}}$ recovers the self-adjoint operators considered by Bismut, Gillet, Lebeau, and Soul\'{e} \cite{B1,BGS1,BGS2,BGS3,BGS4,BL1,BL2}. When the bundle $E$ is unitary flat, the operator $\square_{E,\bar{\partial}}$ recovers the self-adjoint operators of Ray and Singer, \cite{RS2}. For more details about the operator $\square_{E,\bar{\partial}}$, cf. \cite{CM}.

\subsection{Twisted Dolbeault-type cohomology and homology groups}

For each $0 \le p \le n$, denote by $A^{p,\bar{0}}(W,E):=A^{p,\operatorname{even}}(W,E)$ and $A^{p,\bar{1}}(W,E):=A^{p,\operatorname{odd}}(W,E)$. Let $H \in A^{0,\bar{1}}(W,\mathbb{C})$ and $\bar{\partial}_E^H:=\bar{\partial}_E+H \wedge \cdot$. We assume that $\bar{\partial} H=0$, then, as in the de Rham case, $(\bar{\partial}_E^H)^2=0$. Hence, we can consider the following twisted complex:
\[
\big(A^{p,\bullet}(W,E), \bar{\partial}_E^H  \big): \cdots \stackrel{\bar{\partial}_E^H}{\longrightarrow} A^{p,\bar{0}}(W,E)\stackrel{\bar{\partial}_E^H}{\longrightarrow} A^{p,\bar{1}}(W,E)\stackrel{\bar{\partial}_E^H}{\longrightarrow} A^{p,\bar{0}}(W,E)\stackrel{\bar{\partial}_E^H}{\longrightarrow} \cdots.
\]
We define the twisted Dolbeault-type cohomology groups of $\big(A^{p,\bullet}(W,E), \bar{\partial}_E^H  \big)$ as
\[
H^{p,\bar{k}}_{\bar{\partial}_E}(W,E,H):=\frac{\Ker(\bar{\partial}_E^H:A^{p,\bar{k}}(W,E) \to A^{p,\overline{k+1}}(W,E))}{\operatorname{Im}(\bar{\partial}_E^H:A^{p,\overline{k-1}}(W,E) \to A^{p,\overline{k}}(W,E))}, \quad k=0,1.
\]

Denote by $\bar{H}:=\operatorname{conj} H $. Let $D^{1,0}_{\bar{H}}:=D^{1,0}+ \bar{H} \wedge \cdot$, then $(D^{1,0}_{\bar{H}})^2=0$. Hence, we can also consider the following twisted complex:
\[
\big(A^{\bullet, p}(W,E), D^{1,0}_{\bar{H}}  \big): \cdots \buildrel D^{1,0}_{\bar{H}}\over\longrightarrow A^{\bar{0},p}(W,E)\stackrel{D^{1,0}_{\bar{H}}}{\longrightarrow} A^{\bar{1},p}(W,E)\stackrel{D^{1,0}_{\bar{H}}}{\longrightarrow} A^{\bar{0},p}(W,E)\stackrel{D^{1,0}_{\bar{H}}}{\longrightarrow} \cdots.
\]
We define the twisted Dolbeault-type cohomology groups of $\big(A^{p,\bullet}(W,E), D^{1,0}_{\bar{H}}  \big)$ as
\[
H^{\bar{k},p}_{D^{1,0}}(W,E,H):=\frac{\Ker(D^{1,0}_{\bar{H}}:A^{\bar{k},p}(W,E) \to A^{\overline{k+1},p}(W,E))}{\operatorname{Im}(D^{1,0}_{\bar{H}}:A^{\overline{k-1},p}(W,E) \to A^{\overline{k},p}(W,E))}, \quad k=0,1.
\]

Denote by $\bar{\partial}_{E,D^{1,0}}^{*,H}:=-\hat{\star}(D^{1,0}+ \operatorname{conj} H  \wedge \cdot )\hat{\star}=-\hat{\star}D^{1,0}_{\bar{H}} \hat{\star} $, then we have $(\bar{\partial}_E^{*,H})^2=0$. Again we can consider the following twisted complex:
\[
\big(A^{p,\bullet}(W,E), \bar{\partial}_{E,D^{1,0}}^{*,H}  \big): \cdots \stackrel{\bar{\partial}_{E,D^{1,0}}^{*,H}}{\longleftarrow} A^{p,\bar{0}}(W,E)\stackrel{\bar{\partial}_{E,D^{1,0}}^{*,H}}{\longleftarrow} A^{p,\bar{1}}(W,E)\stackrel{\bar{\partial}_{E,D^{1,0}}^{*,H}}{\longleftarrow} A^{p,\bar{0}}(W,E)\stackrel{\bar{\partial}_{E,D^{1,0}}^{*,H}}{\longleftarrow} \cdots.
\]
We define the following twisted Dolbeault-type homology groups of $\big(A^{p,\bullet}(W,E), \bar{\partial}_{E,D^{1,0}}^{*,H}  \big)$ as
\[
H_{\bar{k}}(A^{p,\bullet}(W,E), \bar{\partial}_{E,D^{1,0}}^{*,H}):=\frac{\Ker(\bar{\partial}_{E,D^{1,0}}^{*,H}:A^{p,\bar{k}}(W,E) \to A^{p,\overline{k-1}}(W,E))}{\operatorname{Im}(\bar{\partial}_{E,D^{1,0}}^{*,H}:A^{p,\overline{k+1}}(W,E) \to A^{p,\overline{k}}(W,E))}, \quad k=0,1.
\]

The operator $\hat{\star}$ induces a $\mathbb{C}$-linear isomorphism of the complex $(A^{p,\bullet}(W,E), \bar{\partial}_{E,D^{1,0}}^{*,H})$ to the complex $(A^{n-\bullet,n-p}(W,E),\pm D^{1,0}_{\bar{H}})$. Hence, as the $\mathbb{Z}$-graded case, cf. \cite[P. 151]{CM} or \cite[(2.19)]{LY}, we have the following isomorphism:
\begin{equation}\label{E:hnkpiso}
H^{\overline{n-k},n-p}_{D^{1,0}}(W,E,H) \cong H_{\bar{k}}(A^{p,\bullet}(W,E), \bar{\partial}_{E,D^{1,0}}^{*,H}), \quad k=0,1.
\end{equation}

\subsection{$\zeta$-function and $\zeta$-regularized determinant}
In this subsection we briefly recall some definitions of $\zeta$-regularized determinants of non self-adjoint elliptic operators. See \cite[Section 6]{BK3} for more details. Let $F$ be a complex (respectively, holomorphic) vector bundle over a closed smooth (respectively, complex) manifold $N$. Let $D:C^\infty(N,F) \to C^\infty(N,F)$ be an elliptic differential operator of order $m \ge 1$. Assume that $\theta$ is an Agmon angle, cf. for example, Definition 6.3 of \cite{BK3}. Let $\Pi: L^2(N,F) \to L^2(N,F)$ denote the spectral projection of $D$ corresponding to all nonzero eigenvalues of $D$. The $\zeta$-function $\zeta_\theta(s,D)$ of $D$ is defined as follows
\begin{equation}\label{E:zetdef}
\zeta_\theta(s,D)= \operatorname{Tr} \Pi D_\theta^{-s}, \quad \operatorname{Re} s > \frac{\dim N}{m}.
\end{equation} It was shown by Seeley \cite{Se} (See also \cite{Sh}) that $\zeta_\theta(s,D)$ has a meromorphic extension to the whole complex plane and that $0$ is a regular value of $\zeta_\theta(s,D)$.

\begin{definition}
The $\zeta$-regularized determinant of $D$ is defined by the formula
\[
\Det'_\theta(D):= \exp \Big(\, -\frac{d}{ds}\Big|_{s=0}\zeta_\theta(s,D)\,\Big). 
\]
\end{definition}
We denote by 
\[
\LDet'_\theta(D)= - \frac{d}{ds}\Big|_{s=0}\zeta_\theta(s,D).
\]

Let $Q$ be a $0$-th order pseudo-differential projection, ie. a $0$-th order pseudo-differential operator satisfying $Q^2=Q$. We set
\begin{equation}\label{E:zetadefpro}
\zeta_\theta(s,Q,D)= \operatorname{Tr} Q \Pi D_\theta^{-s}, \quad \operatorname{Re} s > \frac{\dim M}{m}.
\end{equation}
The function $\zeta_\theta(s,Q,D)$ also has a meromorphic extension to the whole complex plane and, by Wodzicki, \cite[Section 7]{Wo1}, it is regular at $0$.

\begin{definition}
Suppose that $Q$ is a $0$-th order pseudo-differential projection commuting with $D$. Then $V:=\operatorname{Im} Q$ is $D$ invariant subspace of $C^\infty(M,E)$. The $\zeta$-regularized determinant of the restriction $D|_V$ of $D$ to $V$ is defined by the formula
\[
\Det'_\theta(D|_V):= e^{\LDet'_\theta(D|_V)},
\]
where 
\begin{equation}\label{E:ldetv}
\LDet'_\theta(D|_V)=-\frac{d}{ds}\Big|_{s=0} \zeta_\theta(s,Q,D).
\end{equation}
\end{definition}

\begin{remark}
The prime in $\Det'_\theta$ and $\LDet'_\theta$ indicates that we ignore the zero eigenvalues of the operator in the definition of the regularized determinant. If the operator is invertible we usually omit the prime and write $\Det_\theta$ and $\LDet_\theta$ instead.
\end{remark}

\subsection{Twisted Cappell-Miller holomorphic torsion}
Note that, for each $0 \le p \le n$, the twisted flat $\bar{\partial}$-Laplacian, defined as
\[
\square_{E,\bar{\partial}}^H:=(\bar{\partial}_E^H+ \bar{\partial}_{E,D^{1,0}}^{*,H})^2,
\] 
maps $A^{p,\bar{k}}(W,E), k=0,1$ into itself. Suppose that $\mathcal{I}$ is an interval of the form $[0,\lambda], (\lambda,\mu]$ or $(\lambda,\infty)(\mu > \lambda \ge 0)$. Denote by $\Pi_{\square_{E,\bar{\partial}}^H,\mathcal{I}}$ the spectral projection of $\square_{E,\bar{\partial}}^H$ corresponding to the set of eigenvalues, whose absolute values lie in $\mathcal{I}$. Set
\[
A^{p,\bar{k}}_{\mathcal{I}}(W,E):= \Pi_{\square_{E,\bar{\partial}}^H,\mathcal{I}} \big( A^{p,\bar{k}}(W,E) \big) \subset A^{p,\bar{k}}(W,E), \quad k=0,1.
\]
If the interval $\mathcal{I}$ is bounded, then, for each $0 \le p \le n$, the space $A^{p,\bar{k}}_{\mathcal{I}}(W,E), k=0,1$ is finite dimensional. Since $\bar{\partial}_E^H$ and $\bar{\partial}_{E,D^{1,0}}^{*,H}$ commute with $\square_{E,\bar{\partial}}^H$, the subspace $A^{p,\bar{k}}_{\mathcal{I}}(W,E)$ is a subcomplex of the twisted bi-graded complex $(A^{p,\bullet}(W,E),\bar{\partial}_E^H,\bar{\partial}_{E,D^{1,0}}^{*,H})$. Clearly, for each $\lambda \ge 0$, the complex $A^{p,\bar{k}}(W,E)$ is doubly acyclic, i.e. $H^{\bar{k}}(A^{p,\bullet}_{(\lambda,\infty)}(W,E),\bar{\partial}^H_E)=0$ and $H_{\bar{k}}(A_{(\lambda,\infty)}^{p,\bullet}(W,E), \bar{\partial}_{E,D^{1,0}}^{*,H})=0$. Since
\[
A^{p,\bar{k}}(W,E) \, = \, A_{[0,\lambda]}^{p,\bar{k}}(W,E) \oplus A_{(\lambda,\infty)}^{p,\bar{k}}(W,E),
\]
we have the isomorphisms
\[
H^{\bar{k}}(A^{p,\bullet}_{[0,\lambda]}(W,E),\bar{\partial}^H_E) \cong H^{p,\bar{k}}_{\bar{\partial}_E}(W,E,H)
\]
and, by \eqref{E:hnkpiso},
\[
H_{\bar{k}}(A_{[0,\lambda]}^{p,\bullet}(W,E), \bar{\partial}_{E,D^{1,0}}^{*,H}) \cong H^{\overline{n-k}}(A^{\bullet,n-p}_{[0,\lambda]}(W,E), \pm D^{1,0}_{\bar{H}}) \cong H^{\overline{n-k},n-p}_{D^{1,0}}(W,E,H).
\]
In particular, we have the following isomorphisms
\begin{equation}\label{E:detacohoiso1}
\Det H^{\bullet}(A^{p,\bullet}_{[0,\lambda]}(W,E),\bar{\partial}^H_E) \cong \Det H^{p,\bullet}_{\bar{\partial}_E}(W,E,H)
\end{equation}
and
\begin{equation}\label{E:dethoiso1}
\Det H_{\bullet}(A_{[0,\lambda]}^{p,\bullet}(W,E), \bar{\partial}_{E,D^{1,0}}^{*,H}) \cong \Det H^{\overline{n-\bullet},n-p}_{D^{1,0}}(W,E,H).
\end{equation}

For any $\lambda \ge 0, 0 \le p \le n$, denote by $\tau_{p,[0,\lambda]}$ the Cappell-Miller torsion of the twisted bi-graded complex $(A_{[0,\lambda]}^{p,\bar{k}}(W,E),\bar{\partial}_E^H,\bar{\partial}_{E,D^{1,0}}^{*,H})$, cf. \eqref{E:tauddstar}. Then, by \eqref{E:detacohoiso1} and \eqref{E:dethoiso1}, we can view $\tau_{p,[0,\lambda]}$ as an element of the determinant line
\begin{align}
\tau_{p,[0,\lambda]} & \in \Det H^{p,\bullet}_{\bar{\partial}_E}(W,E,H) \otimes [\Det H^{\overline{n-\bullet},n-p}_{D^{1,0}}(W,E,H)]^{-1}  \nonumber \\ 
& \cong \Det H^{p,\bullet}_{\bar{\partial}_E}(W,E,H) \otimes [\Det H^{\bullet,n-p}_{D^{1,0}}(W,E,H)]^{(-1)^{n+1}}.
\end{align} 

For each $k=0,1$ and each $0 \le p \le n$, set
\[
A^{p,\bar{k}}_{+,\mathcal{I}}(W,E):= \Ker(\bar{\partial}_E^H \bar{\partial}_{E,D^{1,0}}^{*,H}) \cap A^{p,\bar{k}}_{\mathcal{I}}(W,E),
\]
\[
A^{p,\bar{k}}_{-,\mathcal{I}}(W,E):= \Ker( \bar{\partial}_{E,D^{1,0}}^{*,H}\bar{\partial}_E^H) \cap A^{p,\bar{k}}_{\mathcal{I}}(W,E).
\]
Clearly,
\[
A^{p,\bar{k}}_{\mathcal{I}}(W,E)\, = \, A^{p,\bar{k}}_{+,\mathcal{I}}(W,E) \oplus A^{p,\bar{k}}_{-,\mathcal{I}}(W,E), \ \text{if} \ 0 \notin \mathcal{I}.
\]

Let $\theta \in (0,2\pi)$ be an Agmon angle of the operator $\square_{E,\bar{\partial}}^H$, cf. for example \cite[Section 6]{BK3}. Since the leading symbol of the operator $\square_{E,\bar{\partial}}^H$ is positive definite, the $\zeta$-regularized determinant $\Det_\theta(\bar{\partial}_{E,D^{1,0}}^{*,H}\bar{\partial}_E^H)\big|_{A^{p,\bar{k}}_{+,\mathcal{I}}(W,E)}$ is independent of the choice of the Agmon angle $\theta$ of the operator $\square_{E,\bar{\partial}}^H$.

For any $0 \le \lambda \le \mu \le \infty$, one easily sees that
\begin{align}\label{E:detlammu00}
\prod_{k=0,1} \big(\,\Det_{\theta}(\bar{\partial}_{E,D^{1,0}}^{*,H}\bar{\partial}_E^H)\big|_{A^{p,\bar{k}}_{+,(\lambda,\infty)}(W,E)} \, \big)^{(-1)^k}  = & \Big[\prod_{k=0,1} \big(\,\Det_{\theta}(\bar{\partial}_{E,D^{1,0}}^{*,H}\bar{\partial}_E^H)\big|_{A^{p,\bar{k}}_{+,(\lambda,\mu)}(W,E)} \, \big)^{(-1)^k} \Big] \nonumber \\
& \cdot \Big[\prod_{k=0,1} \big(\,\Det_{\theta}(\bar{\partial}_{E,D^{1,0}}^{*,H}\bar{\partial}_E^H)\big|_{A^{p,\bar{k}}_{+,(\mu, \infty)}(W,E)} \, \big)^{(-1)^k} \Big]  
\end{align}

By Proposition \ref{P:indedec} and \eqref{E:detlammu00}, we know that the element 
\begin{equation}\label{E:tauholopweh}
\tau_{\operatorname{holo},p}(W,E,H):=\tau_{p,[0,\lambda]} \cdot \prod_{k=0,1} \big(\,\Det_{\theta}(\bar{\partial}_{E,D^{1,0}}^{*,H}\bar{\partial}_E^H)\big|_{A^{p,\bar{k}}_{+,(\lambda,\infty)}(W,E)} \, \big)^{(-1)^k}
\end{equation}
is independent of the choice of $\lambda$. It is also independent of the choice of the Agmon angle $\theta \in (0,2\pi)$ of the operator $\square_{E,\bar{\partial}}^H$.

We now define the twisted Cappell-Miller holomorphic torsion.

\begin{definition}\label{D:maindef1}
The non-vanishing element of the determinant
\[
\tau_{\operatorname{holo},p}(W,E,H) \in \Det H^{p,\bullet}_{\bar{\partial}_E}(W,E,H) \otimes [\Det H^{\bullet,n-p}_{D^{1,0}}(W,E,H)]^{(-1)^{n+1}}
\]
defined in \eqref{E:tauholopweh} is called the twisted Cappell-Miller holomorphic torsion. 
\end{definition}

\subsection{Twisted Cappell-Miller holomorphic torsion under metric deformation}

Let $g_u^W, u \in \mathbb{R}$, be a smooth family of Hermitian metrics on the complex manifold $W$. Denote by $\star_u$ the Hodge star operators associated to the metrics $g_u^W$ and denote by $\bar{\partial}_{E,D^{1,0},u}^{*,H}:=-\hat{\star}_u(D^{1,0}+ \operatorname{conj} H  \wedge \cdot )\hat{\star}_u$. Let $\square_{E,\bar{\partial},u}^H=(\bar{\partial}^H_E+\bar{\partial}_{E,D^{1,0},u}^{*,H})^2$ be the flat Laplacian operators associated to the metrics $g_u^W$.  

Fix $u_0 \in \mathbb{R}$ and choose $\lambda \ge 0$ so that there are no eigenvalues of $\square_{E,\bar{\partial},u}^H$ whose absolute values are equal to $\lambda$. Further, assume that $\lambda$ is big enough so that the real parts of eigenvalues of $\square_{E,\bar{\partial},u}^H$ are all greater than $0$. Then there exists $\delta > 0$ such that the same is true for all $u \in (u_0-\delta,u_0+\delta)$. In particular, if we denote by $A^{p,\bullet}_{[0,\lambda],u}(W,E)$ the span of the generalized eigenvectors of $\square_{E,\bar{\partial},u}^H$ corresponding to eigenvalues with absolute value $\le \lambda$, then $\dim A^{p,\bullet}_{[0,\lambda],u}(W,E)$ is independent of $u \in (u_0-\delta,u_0+\delta)$.

For any $\lambda \ge 0, 0 \le p \le n$, denote by $\tau_{p,[0,\lambda],u}$ the Cappell-Miller torsion of the twisted bi-graded complex $(A_{[0,\lambda]}^{p,\bullet}(W,E),\bar{\partial}_E^H,\bar{\partial}_{E,D^{1,0},u}^{*,H})$. Set 
\[
\alpha_u = \star_u^{-1} \cdot \frac{d}{du} \star_u = = \hat{\star}_u^{-1} \cdot \frac{d}{du} \hat{\star}_u.
\]
The proof of the following lemma is similar to the proof of Lemma 7.1 of \cite{CM}, where the untwisted case was treated. To save space, we omit the proof.
\begin{lemma}\label{L:holvarlem1}
Under the above assumptions, we have 
\[
\frac{d}{du}\tau_{p,[0,\lambda],u} = -\sum_{k=0,1}(-1)^k  \operatorname{Tr} \big[ \alpha_u \big|_{A^{p,\bar{k}}_{[0,\lambda]}(W,E)} \big] \cdot \tau_{p,[0,\lambda],u}.
\]
\end{lemma}

We also need the following lemma.
\begin{lemma}\label{L:holvarlem2}
Under the above assumptions, we have
\[
\frac{d}{du}\Big[ \sum_{k=0,1}(-1)^k \LDet_{\theta}(\bar{\partial}_{E,D^{1,0},u}^{*,H}\bar{\partial}_E^H)\big|_{A^{p,\bar{k}}_{+,(\lambda,\infty)}(W,E)} \Big] \, = \,  \sum_{k=0,1}(-1)^k  \operatorname{Tr} \big[ \alpha_u \big|_{A^{p,\bar{k}}_{[0,\lambda]}(W,E)} \big] + \sum_{k=0,1}(-1)^k  \int_W b_{n,p,\bar{k},u},
\]
where $b_{n,p,\bar{k},u}$ is given by a local formula.
\end{lemma}
\begin{proof}
Set 
\begin{align}
f(s,u) & =\sum_{k=0,1}(-1)^k\int_0^\infty t^{s-1} \operatorname{Tr} \Big( \exp\big(-t (\bar{\partial}_{E,D^{1,0},u}^{*,H}\bar{\partial}_E^H)\big|_{A^{p,\bar{k}}_{+,(\lambda,\infty)}(W,E)} \big)  \Big)dt \nonumber \\
 & = \Gamma(s) \sum_{k=0,1}(-1)^k \zeta \big(s, (\bar{\partial}_{E,D^{1,0},u}^{*,H}\bar{\partial}_E^H)\big|_{A^{p,\bar{k}}_{+,(\lambda,\infty)}(W,E)} \big).
\end{align}

Then we have 
\[
\frac{d}{du} \bar{\partial}_{E,D^{1,0},u}^{*,H}\big|_{A^{p,\overline{k+1}}_{-,(\lambda,\infty)}(W,E)}=-\big[\alpha_u, \bar{\partial}_{E,D^{1,0},u}^{*,H}\big|_{A^{p,\overline{k+1}}_{-,(\lambda,\infty)}(W,E)}\big],
\]
which follows easily from $\bar{\partial}_{E,D^{1,0}}^{*,H}:=-\hat{\star}(D^{1,0}+ \operatorname{conj} H  \wedge \cdot )\hat{\star}$ and the equality $\star_u^{-1} \cdot \frac{d}{du} \star_u = - \frac{d}{du} \star_u \cdot \star_u^{-1}$.

If $A$ is of trace class and $B$ is a bounded operator, it is well known that $\operatorname{Tr}(AB)=\operatorname{Tr}(BA)$. Hence, by this fact and the semi-group property of the heat operator, we have
\begin{align}\label{E:traidenti1}
 & \operatorname{Tr} \Big( \bar{\partial}_{E,D^{1,0},u}^{*,H}\big|_{A^{p,\overline{k+1}}_{-,(\lambda,\infty)}(W,E)} \alpha_u \bar{\partial}^H_E \big|_{A^{p,\overline{k}}_{+,(\lambda,\infty)}(W,E)}\exp\big(-t (\bar{\partial}_{E,D^{1,0},u}^{*,H}\bar{\partial}_E^H)\big|_{A^{p,\bar{k}}_{+,(\lambda,\infty)}(W,E)}\big) \Big) \nonumber \\
= & \operatorname{Tr} \Big(\exp\big(-\frac{t}{2} (\bar{\partial}_{E,D^{1,0},u}^{*,H}\bar{\partial}_E^H)\big|_{A^{p,\bar{k}}_{+,(\lambda,\infty)}(W,E)}\big) \bar{\partial}_{E,D^{1,0},u}^{*,H}\big|_{A^{p,\overline{k+1}}_{-,(\lambda,\infty)}(W,E)} \nonumber \\
& \cdot \alpha_u \bar{\partial}^H_E \big|_{A^{p,\overline{k}}_{+,(\lambda,\infty)}(W,E)}\exp\big(-\frac{t}{2} (\bar{\partial}_{E,D^{1,0},u}^{*,H}\bar{\partial}_E^H)\big|_{A^{p,\bar{k}}_{+,(\lambda,\infty)}(W,E)}\big) \Big) \nonumber \\
= & \operatorname{Tr} \Big( \alpha_u \bar{\partial}^H_E \big|_{A^{p,\overline{k}}_{+,(\lambda,\infty)}(W,E)}\exp\big(-\frac{t}{2} (\bar{\partial}_{E,D^{1,0},u}^{*,H}\bar{\partial}_E^H)\big|_{A^{p,\bar{k}}_{+,(\lambda,\infty)}(W,E)}\big) \nonumber \\ & \cdot \exp \big(-\frac{t}{2} (\bar{\partial}_{E,D^{1,0},u}^{*,H}\bar{\partial}_E^H)\big|_{A^{p,\bar{k}}_{+,(\lambda,\infty)}(W,E)} \big) \bar{\partial}_{E,D^{1,0},u}^{*,H}\big|_{A^{p,\overline{k+1}}_{-,(\lambda,\infty)}(W,E)} \Big) \nonumber \\
= & \operatorname{Tr} \Big( \alpha_u (\bar{\partial}^H_E \bar{\partial}_{E,D^{1,0},u}^{*,H})\big|_{A^{p,\overline{k+1}}_{-,(\lambda,\infty)}(W,E)}\exp\big(-t (\bar{\partial}^H_E \bar{\partial}_{E,D^{1,0},u}^{*,H})\big|_{A^{p,\overline{k+1}}_{-,(\lambda,\infty)}(W,E)}\big)  \Big)
\end{align}
Now
\begin{align}\label{E:dduidentiti1}
\frac{d}{du} f(s,u) & = \sum_{k=0,1}(-1)^k\int_0^\infty t^{s-1} \operatorname{Tr} \Big( t\big[\alpha_u, \bar{\partial}_{E,D^{1,0}}^{*,H}\big|_{A^{p,\overline{k+1}}_{-,(\lambda,\infty)}(W,E)}\big]  \exp\big(-t (\bar{\partial}_{E,D^{1,0}}^{*,H}\bar{\partial}_E^H)\big|_{A^{p,\bar{k}}_{+,(\lambda,\infty)}(W,E)}\big)  \Big)dt \nonumber \\
= & \sum_{k=0,1}(-1)^k\int_0^\infty t^{s-1} \operatorname{Tr} \Big(\, t\alpha_u \big(\, (\bar{\partial}_{E,D^{1,0}}^{*,H}\bar{\partial}_E^H)\big|_{A^{p,\bar{k}}_{+,(\lambda,\infty)}(W,E)}\exp\big(-t (\bar{\partial}_{E,D^{1,0}}^{*,H}\bar{\partial}_E^H)\big|_{A^{p,\bar{k}}_{+,(\lambda,\infty)}(W,E)}\big) \nonumber \\
& - (\bar{\partial}^H_E \bar{\partial}_{E,D^{1,0},u}^{*,H})\big|_{A^{p,\overline{k+1}}_{-,(\lambda,\infty)}(W,E)}\exp\big(-t (\bar{\partial}^H_E \bar{\partial}_{E,D^{1,0},u}^{*,H})\big|_{A^{p,\overline{k+1}}_{-,(\lambda,\infty)}(W,E)}\big) \, \big) \, \Big) dt, \quad \text{by} \, \eqref{E:traidenti1}  \nonumber \\
= & \sum_{k=0,1}(-1)^k\int_0^\infty t^{s} \operatorname{Tr} \Big(\, \alpha_u  \square_{E,\bar{\partial},u}^H\big|_{A^{p,\bar{k}}_{(\lambda,\infty)}(W,E)}\exp\big(-t \square_{E,\bar{\partial},u}^H\big|_{A^{p,\bar{k}}_{(\lambda,\infty)}(W,E)}\big) \, \Big) dt \nonumber \\
= & -\sum_{k=0,1}(-1)^k\int_0^\infty t^{s} \frac{d}{dt} \operatorname{Tr} \Big(\, \alpha_u  \exp\big(-t \square_{E,\bar{\partial},u}^H\big|_{A^{p,\bar{k}}_{(\lambda,\infty)}(W,E)}\big) \, \Big) dt. \nonumber \\
= & s \sum_{k=0,1}(-1)^k\int_0^\infty t^{s-1}  \operatorname{Tr} \Big(\, \alpha_u  \exp\big(-t \square_{E,\bar{\partial},u}^H\big|_{A^{p,\bar{k}}_{(\lambda,\infty)}(W,E)}\big) \, \Big) dt,
\end{align}
where we used integration by parts for the last equality. Since $\square_{E,\bar{\partial},u}^H$ is an elliptic operator, the dimension of $A^{p,\bullet}_{[0,\lambda]}(W,E)$ is finite. Let $\varepsilon \not= 0$ be a small enough real number so that $\square_{E,\bar{\partial},u}^H + \varepsilon$ is bijective. Then we can rewite \eqref{E:dduidentiti1} as
\begin{align}\label{E:dduidentiti2}
\frac{d}{du} f(s,u) & = s \sum_{k=0,1}(-1)^k\int_0^1 t^{s-1}  \operatorname{Tr} \Big(\, \alpha_u  \exp\big(-t (\square_{E,\bar{\partial},u}^H + \varepsilon)\big|_{A^{p,\bar{k}}(W,E)}\big) \, \Big) dt \nonumber \\
& + s \sum_{k=0,1}(-1)^k\int_1^\infty t^{s-1}  \operatorname{Tr} \Big(\, \alpha_u  \exp\big(-t (\square_{E,\bar{\partial},u}^H + \varepsilon)\big|_{A^{p,\bar{k}}(W,E)} \big) \, \Big) dt \nonumber \\
& - s \sum_{k=0,1}(-1)^k\int_0^1 t^{s-1}  \operatorname{Tr} \Big(\, \alpha_u  \exp\big(-t \square_{E,\bar{\partial},u}^H\big|_{A^{p,\bar{k}}_{[0,\lambda]}(W,E)}\big) \, \Big) dt \nonumber \\
& - s \sum_{k=0,1}(-1)^k\int_1^\infty t^{s-1}  \operatorname{Tr} \Big(\, \alpha_u  \exp\big(-t \square_{E,\bar{\partial},u}^H\big|_{A^{p,\bar{k}}_{[0,\lambda]}(W,E)}\big) \, \Big) dt. 
\end{align}
Now $\dim W = 2n$ is even, so for small time asymptotic exapnsion for $\operatorname{Tr} \Big(\, \alpha_u  \exp\big(-t (\square_{E,\bar{\partial},u}^H + \varepsilon)\big|_{A^{p,\bar{k}}(W,E)}\big) \, \Big)$ has a term $a_{n,p,\bar{k},u}t^0$ in its expansion about $t=0$. That means $\operatorname{Tr} \Big(\, \alpha_u  \exp\big(-t (\square_{E,\bar{\partial},u}^H + \varepsilon)\big|_{A^{p,\bar{k}}(W,E)}\big) \, \Big)-a_{n,p,\bar{k},u}t^0$ does not contain a constant term as $t \downarrow 0$. Hence, the integrals
\[
\sum_{k=0,1}(-1)^k\int_0^1 t^{s-1}  \operatorname{Tr} \Big(\, \alpha_u  \exp\big(-t (\square_{E,\bar{\partial},u}^H + \varepsilon)\big|_{A^{p,\bar{k}}(W,E)}\big)  \, \Big)-a_{n,p,\bar{k},u}t^0 dt
\]
do not have poles at $s=0$. But the integrals
\[
\sum_{k=0,1}(-1)^k\int_0^1 t^{s-1} a_{n,p,\bar{k},u}t^0  dt
\]
have poles of order $1$ with residues $a_{n,p,\bar{k},u}, k=0,1$. On the other hand, because of exponential decay of $\operatorname{Tr} \Big(\, \alpha_u  \exp\big(-t (\square_{E,\bar{\partial},u}^H + \varepsilon)\big|_{A^{p,\bar{k}}(W,E)}\big) \, \Big)$ and $\operatorname{Tr} \Big(\, \alpha_u  \exp\big(-t \square_{E,\bar{\partial},u}^H\big|_{A^{p,\bar{k}}_{[0,\lambda]}(W,E)}\big) \, \Big)$ for large $t$, the integrals of the second term and the fourth term on the right hand side of \eqref{E:dduidentiti2} are entire functions in $s$. Hence we have
\begin{align}
\frac{d}{du}\big|_{s=0} f(s,u) & =  -s\Big(\, \sum_{k=0,1}(-1)^k \int_0^1 t^{s-1} \big(\operatorname{Tr} \big[\alpha_u \big|_{A^{p,\bar{k}}_{[0,\lambda]}(W,E)}\big] - a_{n,p,\bar{k},u} \big)dt \, \Big) \big|_{s=0}  \nonumber \\
& = - \sum_{k=0,1}(-1)^k  \operatorname{Tr} \big[ \alpha_u \big|_{A^{p,\bar{k}}_{[0,\lambda]}(W,E)} \big] + \sum_{k=0,1}(-1)^k a_{n,p,\bar{k},u}.
\end{align}
Hence, the result follows.
\end{proof}

By combining Lemma \ref{L:holvarlem1} with Lemma \ref{L:holvarlem2}, we obtain the main theorem of this subsection. For untwisted case, cf. \cite[Theorem 4.4]{CM}. 
\begin{theorem}\label{T:vartheorem1}
Let $W$ be a complex manifold of complex dimension $n$ and let $E$ is a holomorphic bundle with connection $D$ that is compatible and type $(1,1)$ over $W$. Suppose that $H \in A^{0,\bar{1}}(W,\mathbb{C})$ and $\bar{\partial} H=0$. Let $g_u^W, u \in (u_0-\delta,u_0+\delta)$, be a smooth family of Riemannian metrics on the complex manifold $W$, then the corresponding twisted Cappell-Miller holomorphic torsion $\tau_{\operatorname{holo},p,u}(W,E,H)$ varies smoothly and the variation of $\tau_{\operatorname{holo},p,u}(W,E,H)$ is given by a local formula
\[
\frac{d}{du} \tau_{\operatorname{holo},p,u}(W,E,H) \, = \, \big(\sum_{k=0,1}(-1)^k  \int_W b_{n,p,\bar{k},u}\big) \cdot \tau_{\operatorname{holo},p,u}(W,E,H).
\] 
\end{theorem}

We have the following corollary. See also \cite[Theorem 5.3, Corollary 7.1]{MW3} for the case of analytic torsion on $\mathbb{Z}_2$-graded elliptic complexes.
\begin{corollary}\label{C:relavar1}
Let $W$ be a complex manifold of complex dimension $n$ and let $E$ is a holomorphic bundle with connection $D$ that is compatible and type $(1,1)$ over $W$. Suppose that $H \in A^{0,\bar{1}}(W,\mathbb{C})$ and $\bar{\partial} H=0$. Let $F_1, F_2$ be two flat complex bundles over $W$ of the same dimension, then 
\[
\tau_{\operatorname{holo},p}(W,E \otimes F_1,H) \otimes [\tau_{\operatorname{holo},p}(W,E \otimes F_2 ,H)]^{-1}
\]
in the tensor product of determinant lines
\begin{align}
\big(\Det H^{p,\bullet}_{\bar{\partial}_E}(W,E \otimes F_1,H) \otimes [\Det H^{\bullet,n-p}_{D^{1,0}}(W,E \otimes F_1,H)]^{(-1)^{n+1}}\big) & \nonumber \\  \otimes \big( \Det H^{p,\bullet}_{\bar{\partial}_E}(W,E \otimes F_2 ,H) & \otimes [\Det H^{\bullet,n-p}_{D^{1,0}}(W,E \otimes F_2,H)]^{(-1)^{n+1}} \big)^{-1} \nonumber
\end{align}
is independent of the Hermitian metric $g^{W}$ chosen.
\end{corollary}

Since the two bundles $E \otimes F_1$ and $E \otimes F_2$ are locally identical as bundles with connections, so the local correction terms are the same and cancell out in computing the variation of the twisted Cappell-Miller holomorphic torsion as the untwisted case, cf. \cite[Corollary 4.5]{CM}.

\subsection{Twisted Cappell-Miller holomorphic torsion under flux deformation}
Suppose that the flux form $H$ is deformed smoothly along a one-parameter family with parameter $v \in \mathbb{R}$ in such a way that the cohomology class $[H] \in H^{0,\bar{1}}(W,\mathbb{C})$ is fixed. Then $\frac{d}{dv}H=-\bar{\partial}B$ for some form $B \in A^{0,\bar{0}}(W,\mathbb{C})$ that depends smoothly on $v$. Let $\beta = B \wedge \cdot$. Fix $v_0 \in \mathbb{R}$ and choose $\lambda > 0$ such that there are no eigenvalues of $\square_{E,\bar{\partial},v_0}^H$ of absolute value $\lambda$ and the real parts of the eigenvalues of $\square_{E,\bar{\partial},v_0}^H\big|_{A^{p,\bar{k}}_{(\lambda,\infty)}(W,E)}$ are all greater than $0$. Then there exists $\delta > 0$ small enough that the same holds for the spectrum of $\square_{E,\bar{\partial},v}^H\big|_{A^{p,\bar{k}}_{(\lambda,\infty)}(W,E)}$ for $v \in (v_0-\delta, v_0+\delta)$. For simplicity, we omit the parameter $v$ in the notations in the following discussion.

The proof of the following lemma is similar to the proof of \cite[Lemma 3.7]{MW}, see also \cite[Lemma 4.7]{H1}. We omit the proof.
\begin{lemma}\label{L:holvarlem3}
Under the above assumptions, we have 
\[
\frac{d}{dv}\tau_{p,[0,\lambda]} = -\sum_{k=0,1}(-1)^k  \operatorname{Tr} \big[ \beta \big|_{A^{p,\bar{k}}_{[0,\lambda]}(W,E)} \big] \cdot \tau_{p,[0,\lambda]},
\]
upon identification of determinant lines under the deformation.
\end{lemma}

We also need the following lemma.
\begin{lemma}\label{L:holvarlem4}
Under the above assumptions, we have
\[
\frac{d}{dv}\Big[ \sum_{k=0,1}(-1)^k \LDet_{\theta}(\bar{\partial}_{E,\nabla^{1,0}}^{*,H}\bar{\partial}_E^H)\big|_{A^{p,\bar{k}}_{+,(\lambda,\infty)}(W,E)} \Big] \, = \,  \sum_{k=0,1}(-1)^k  \operatorname{Tr} \big[ \beta \big|_{A^{p,\bar{k}}_{[0,\lambda]}(W,E)} \big] + \sum_{k=0,1}(-1)^k  \int_W c_{n,p,\bar{k}},
\]
where $c_{n,p,\bar{k}}$ is given by a local formula.
\end{lemma}
\begin{proof}
Under the deformation, we have
\[
\frac{d}{dv} \bar{\partial}_E^H =[\beta, \bar{\partial}_E^H], \quad \frac{d}{dv} \bar{\partial}_{E,D^{1,0}}^{*,H} = - [\beta, \bar{\partial}_{E,D^{1,0}}^{*,H}].
\]
Following the proof of Lemma 3.5 of \cite{MW}, we obtain the desired variation formula.
\end{proof}

By combining Lemma \ref{L:holvarlem3} with Lemma \ref{L:holvarlem4}, we obtain the main theorem of this subsection. 
\begin{theorem}\label{T:vartheorem2}
Let $W$ be a complex manifold of complex dimension $n$ and let $E$ is a holomorphic bundle with connection $D$ that is compatible and type $(1,1)$ over $W$. Along any one parameter deformation of $H$ that fixes the cohomology class $[H]$ and the natural identification of determinant lines, we have the following variation formula
\[
\frac{d}{dv} \tau_{\operatorname{holo},p}(W,E,H) \, = \, \big(\sum_{k=0,1}(-1)^k  \int_W c_{n,p,\bar{k}}\big) \cdot \tau_{\operatorname{holo},p}(W,E,H).
\] 
\end{theorem}

As Corollary \ref{C:relavar1}, we have the following corollary. See also \cite[Corollary 7.1]{MW3} for the case of analytic torsion on $\mathbb{Z}_2$-graded elliptic complexes.
\begin{corollary}
Let $W$ be a complex manifold of complex dimension $n$ and let $E$ is a holomorphic bundle with connection $D$ that is compatible and type $(1,1)$ over $W$. Suppose that $H \in A^{0,\bar{1}}(W,\mathbb{C})$ and $\bar{\partial} H=0$. Let $F_1, F_2$ be two flat complex bundles over $W$ of the same dimension, then $\tau_{\operatorname{holo},p}(W,E \otimes F_1,H) \otimes [\tau_{\operatorname{holo},p}(W,E \otimes F_2 ,H)]^{-1}$ is invariant under any deformation of $H$ by an $\bar{\partial}$-exact form, up to natural identification of the determinant lines.
\end{corollary}

\section{Twisted Cappell-Miller analytic torsion}

In this section we first define the de Rham bi-graded complex twisted by a flux form $H$ and its (co)homology groups. Then we define the Cappell-Miller analytic torsion for the twisted de Rham bi-graded complex. We obtain the variation theorems of the twisted Cappell-Miller analytic torsion under metric and flux deformations. In a recent preprint of \cite{Su}, Su also briefly discussed the twisted Cappell-Miller analytic torsion when dimension of the manifold $M$ is odd.

\subsection{The twisted de Rham complexes} Let $M$ be a closed oriented $m$-dimensional smooth manifold and let $\mathcal{E}$ be a complex vector bundle over $M$ endowed with a flat connection $\nabla$. We denote by $\Omega^p(M,\mathcal{E})$ the space of $p$-forms with values in the flat bundle $\mathcal{E}$, i.e., $\Omega^p(M,\mathcal{E})=\Gamma(M,\wedge^p(T^*M)_{\mathbb{R}}\otimes \mathcal{E})$ and by
\[
\nabla : \Omega^\bullet(M,\mathcal{E}) \to \Omega^{\bullet+1}(M,\mathcal{E})
\]
the covariant differential induced by the flat connection on $\mathcal{E}$. Fix a Riemannian metric $g^M$ on $M$ and let $\star : \Omega^\bullet(M,\mathcal{E}) \to \Omega^{m-\bullet}(M,\mathcal{E})$ denote the Hodge $\star$-operator. We choose a Hermitian metric $h^\mathcal{E}$ so that together with the Riemannian metric $g^M$ we can define a scalar product $<\cdot,\cdot>_M$ on $\Omega^{ \bullet}(M,\mathcal{E})$. Define the chirality operator $\Gamma = \Gamma(g^M):\Omega^\bullet(M,\mathcal{E}) \to \Omega^\bullet(M,\mathcal{E})$ by the formula, cf. \cite[(7-1)]{BK3},
\begin{equation}\label{E:chirality}
\Gamma \omega := i^r(-1)^{\frac{q(q+1)}{2}} \star \omega, \quad \omega \in \Omega^q(M,\mathcal{E}), 
\end{equation}
where $r$ given as above by $r=\frac{m+1}{2}$ if $m$ is odd and $r=\frac{m}{2}$ if $m$ is even. The numerical factor in \eqref{E:chirality} has been chosen so that $\Gamma^2=\operatorname{Id}$, cf. Proposition 3.58 of \cite{BGV}.

Assume that $\mathcal{H}$ is an odd degree closed differential form on $M$. Let $\Omega^{\bar{0}/\bar{1}}(M,\mathcal{E}):=\Omega^{\even/\odd}(M,\mathcal{E})$ and $\nabla^\mathcal{H}:=\nabla+\mathcal{H} \wedge \cdot$. We assume that $\mathcal{H}$ does not contain a 1-form component, which can be absorbed in the flat connection $\nabla$.

It is not difficult to check that $(\nabla^\mathcal{H})^2=0$.
Clearly, for each $k=0,1$, $\nabla^\mathcal{H} : \Omega^{\bar k}(M,\mathcal{E}) \to \Omega^{\overline{k+1}}(M,\mathcal{E})$. Hence we can consider the following twisted de Rham complex:
\begin{equation}\label{E:ztcomplex}
\big(\, \Omega^{  \bullet}(M,\mathcal{E}), \nabla^\mathcal{H} \, \big): \cdots \stackrel{\nabla^\mathcal{H}}{\longrightarrow} \Omega^{\bar 0}(M,\mathcal{E} ) \stackrel{\nabla^\mathcal{H}}{\longrightarrow} \Omega^{\bar 1}(M,\mathcal{E} ) \stackrel{\nabla^\mathcal{H}}{\longrightarrow} \Omega^{\bar 0}(M,\mathcal{E} )\stackrel{\nabla^\mathcal{H}}{\longrightarrow} \cdots.
\end{equation}
We define the {\em twisted de Rham cohomology group of $(\Omega^{\bullet}(M,\mathcal{E}),\nabla^\mathcal{H})$} as
\[
H^{\bar{k}}(M,\mathcal{E},\mathcal{H})  \, := \, \frac{\Ker(\nabla^\mathcal{H} : \Omega^{\bar{k}}(M,\mathcal{E}) \to \Omega^{\overline{k+1}}(M,\mathcal{E}))}{\im (\nabla^\mathcal{H} : \Omega^{\overline{k-1}}(M,\mathcal{E}) \to \Omega^{\bar{k}}(M,\mathcal{E}))}, \quad k=0,1.
\] 
The groups $H^{\bar{k}}(M,\mathcal{E},\mathcal{H}),k=0,1$ are independent of the choice of the Riemannian metric on $M$ or the Hermitian metric on $\mathcal{E}$. Suppose that $\mathcal{H}$ is repalced by $\mathcal{H}'=\mathcal{H}-d\mathcal{B}$ for some $\mathcal{B} \in \Omega^{\bar{0}}(M)$, then there is an isomorphism $\varepsilon_\mathcal{B}:=e^\mathcal{B} \wedge \cdot : \Omega^{  \bullet}(M,\mathcal{E}) \to \Omega^{  \bullet}(M,\mathcal{E})$ satisfying 
\[
\varepsilon_\mathcal{B} \circ \nabla^\mathcal{H} = \nabla^{\mathcal{H}'} \circ \varepsilon_\mathcal{B}.
\]
Therefor $\varepsilon_\mathcal{B}$ induces an isomorphism on the twisted de Rham cohomology, also denote by $\varepsilon_\mathcal{B}$,
\begin{equation}\label{E:epsiso}
\varepsilon_\mathcal{B}: H^\bullet(M,\mathcal{E},\mathcal{H}) \to H^\bullet(M,\mathcal{E},\mathcal{H}').
\end{equation}
%Denote by $\nabla^H_{\bar k}$ the restriction of $\nabla^H$ to $\Omega^{\bar{k}}(M,E)$. 
Denote by $\nabla^{\mathcal{H},*}$ the adjoint of $\nabla^\mathcal{H}$ with respect to the scalar product $<\cdot,\cdot>_M$. Then the Laplacian
\[
\Delta^\mathcal{H}:=\nabla^{\mathcal{H},*}\nabla^\mathcal{H} + \nabla^\mathcal{H}\nabla^{\mathcal{H},*},
\]
is an elliptic operator and therefore the complex \eqref{E:ztcomplex} is elliptic. By Hodge theory, we have the isomorphism $\Ker \Delta^\mathcal{H} \cong H^{\bullet}(M,\mathcal{E},\mathcal{H})$. For more details of the twisted de Rham cohomology, cf. for example \cite{MW}.

Now denote by $\nabla'$ the connection on $\mathcal{E}$ dual to the connection $\nabla$, \cite[Subsection 10.1]{BK3}. We denote by $\mathcal{E}'$ the flat bundle $(\mathcal{E},\nabla')$, referring to $\mathcal{E}'$ as the dual of the flat vector bundle $\mathcal{E}$. We emphasize that, similar to the untwisted case, cf. \cite[(10-8)]{BK3} or \cite[(8.4)]{CM}, 
\[
\nabla^{\mathcal{H},*} = \Gamma {\nabla'}^\mathcal{H} \Gamma,
\]
where ${\nabla'}^\mathcal{H}=\nabla' + \mathcal{H} \wedge \cdot$.

Let $\nablahsharp:= \Gamma \nabla^\mathcal{H}  \Gamma$, then $(\nablahsharp)^2=0$. Clearly, $ \nablahsharp : \Omega^{\bar k}(M,\mathcal{E}) \to \Omega^{\overline{k-1}}(M,\mathcal{E})$. Hence we can consider the following twisted de Rham complex:
\begin{equation}\label{E:ztcomplex21}
\big(\, \Omega^{  \bullet}(M,\mathcal{E}), \nablahsharp \, \big): \cdots \stackrel{\nablahsharp}{\longleftarrow} \Omega^{\bar 0}(M,\mathcal{E}) \stackrel{\nablahsharp}{\longleftarrow} \Omega^{\bar 1}(M,\mathcal{E} ) \stackrel{\nablahsharp}{\longleftarrow} \Omega^{\bar 0}(M,\mathcal{E} )\stackrel{\nablahsharp}{\longleftarrow} \cdots.
\end{equation} 
We also define the homology group of the complex $(\Omega^{\bullet}(M,\mathcal{E}),\nablahsharp)$ as 
\[
H_{\bar{k}}(\Omega^{\bullet}(M,\mathcal{E}),\nablahsharp) \, := \, \frac{\operatorname{Ker}(\nablahsharp : \Omega^{\bar{k}}(M,\mathcal{E}) \to \Omega^{\overline{k-1}}(M,\mathcal{E}))}{\operatorname{Im} (\nablahsharp : \Omega^{\overline{k+1}}(M,\mathcal{E}) \to \Omega^{\bar{k}}(M,\mathcal{E}))}, \quad k=0,1.
\]
Similarly, the groups $H_{\bar{k}}(\Omega^{\bullet}(M,\mathcal{E}),\nablahsharp),k=0,1$ are independent of the choice of the Riemannian metric on $M$ or the Hermitian metric on $\mathcal{E}$. Suppose that $\mathcal{H}$ is repalced by $\mathcal{H}''=\mathcal{H}-\delta \mathcal{B}'$ for some $\mathcal{B}' \in \Omega^{\bar{0}}(M)$ and $\delta$ the adjoint of $d$ with respect to the scalar product induced by the Riemannian metric $g^M$, then there is an isomorphism $\varepsilon_{\mathcal{B}'}:=e^{\mathcal{B}'} \wedge \cdot : \Omega^{\bullet}(M,\mathcal{E}) \to \Omega^{ \bullet}(M,\mathcal{E})$ satisfying 
\[
\varepsilon_{\mathcal{B}'} \circ \nabla^{\mathcal{H},\sharp} = \nabla^{\mathcal{H}'',\sharp} \circ \varepsilon_{\mathcal{B}'}.
\]
Therefor $\varepsilon_{\mathcal{B}'}$ induces an isomorphism on the twisted de Rham homology, also denote by $\varepsilon_{\mathcal{B}'}$,
\begin{equation}\label{E:epsiso12}
\varepsilon_{\mathcal{B}'}: H_\bullet(\Omega^{\bullet}(M,\mathcal{E}),\nablahsharp) \to H_\bullet(\Omega^{\bullet}(M,\mathcal{E}),\nabla^{\mathcal{H}'',\sharp}).
\end{equation}
%Denote by $\nabla^{H,\sharp}_{\bar k}$ the restriction of $\nabla^{H,\sharp}$ to $\Omega^{\bar{k}}(M,E)$.
Denote by $\nabla^{\mathcal{H},\sharp,*}$ the adjoint of $\nabla^{\mathcal{H},\sharp}$ with respect to the scalar product $<\cdot,\cdot>_M$, then we have the following equalities.
\[
\nabla^{\mathcal{H},\sharp,*}={\nabla'}^\mathcal{H}, \qquad {\Delta'}^\mathcal{H}:={\nabla'}^{\mathcal{H},*}{\nabla'}^\mathcal{H} + {\nabla'}^\mathcal{H}{\nabla'}^{\mathcal{H},*}=\nabla^{\mathcal{H},\sharp}\nabla^{\mathcal{H},\sharp,*}+\nabla^{\mathcal{H},\sharp,*}\nabla^{\mathcal{H},\sharp}.
\]
Again the Laplacian ${\Delta'}^\mathcal{H}$ is an elliptic operator and therefore the complex \eqref{E:ztcomplex21} is elliptic. By Hodge theory, we have the isomorphism $\Ker {\Delta'}^\mathcal{H} \cong H_{\bullet}(\Omega^{\bullet}(M,\mathcal{E}),\nablahsharp)$. In particular, for $k=0,1$,
\begin{equation}\label{E:hocodurel1p}
H_{\bar k}(\Omega^{\bullet}(M,\mathcal{E}),\nablahsharp) \cong H^{\bar k}(M,\mathcal{E}',\mathcal{H}).
\end{equation}

%%%%%%%%%%%%%%%%%%%%%%%%%%%%%%%%%%%%%%%%%%%%%%%%%%%%%%%%%%%%%%%%%%%%%%%%%%%%%%%%%%%%%%

\subsection{Definition of twisted Cappell-Miller analytic torsion}
%Following the notations of \cite{MW}, let $C^{\bar{k}}:=\Omega^{\bar{k}}(M,E)$ and let $D_{\bar{k}}=D_{\bar{k}}^H$ be the operator $\nabla^H$ acting on $C^{\bar{k}}$ ($k=0,1$). We also let $D^\sharp_{\bar{k}}=D_{\bar{k}}^{H,\sharp}$ be the operator $\nablahsharp$ acting on $C^{\bar{k}}$ ($k=0,1$). 

Note that the flat Laplacian, defined as $$\Delta^{\mathcal{H},\sharp}:=(\nabla^\mathcal{H}+\nablahsharp)^2,$$ maps $\Omega^{\bar{k}}(M,\mathcal{E})$ into itself. Suppose $\mathcal{I}$ is an interval of the form $[0,\lambda],(\lambda,\mu]$, or $(\lambda,\infty)$ $(\mu>\lambda\ge0)$. Denote by $\Pi_{\Delta^{\mathcal{H},\sharp},\mathcal{I}}$ the spectral projection of $\Delta^{\mathcal{H},\sharp}$ corresponding to the set of eigenvalues, whose absolute values lie in $\mathcal{I}$. Set \begin{equation}\label{E:OmecalI}\notag     
\Omega^{\bar{k}}_{\mathcal{I}}(M,\mathcal{E})\ := \ \Pi_{\Delta^{\mathcal{H},\sharp},\mathcal{I}}\big(\, \Omega^{\bar{k}}(M,\mathcal{E})\,\big)\ \subset\  \Omega^{\bar{k}}(M,\mathcal{E}). 
\end{equation}
If the interval $\mathcal{I}$ is bounded, then the space $\Omega^{\bar{k}}_\mathcal{I}(M,\mathcal{E})$ is finite dimensional.
Since $\nabla^\mathcal{H}$ and $\nablahsharp$ commute with $\Delta^\sharp_\mathcal{H}$, the subspace $\Omega^{\bullet}_\mathcal{I}(M,\mathcal{E})$ is a subcomplex of the twisted de Rham bi-complex $(\Omega^{\bullet}(M,\mathcal{E}),\nabla^\mathcal{H}, \nablahsharp)$. Clearly, for each $\lambda\ge 0$, the complex $\Omega^{\bullet}_{(\lambda,\infty)}(M,\mathcal{E})$ is doubly acyclic, i.e. $H^{\bar k}(\Omega^{\bullet}_{(\lambda,\infty)}(M,\mathcal{E}),\nabla^\mathcal{H})=0$ and $H_{\bar k}(\Omega^{\bullet}_{(\lambda,\infty)}(M,\mathcal{E}),\nabla^{\mathcal{H},\sharp})=0$. Since
\begin{equation}\label{E:Ome=Ome0+Ome>0}
    \Omega^{\bar{k}}(M,\mathcal{E}) \ = \  \Omega^{\bar{k}}_{[0,\lambda]}(M,\mathcal{E})\,\oplus\,\Omega^{\bar{k}}_{(\lambda,\infty)}(M,\mathcal{E}),
\end{equation}
the homology $H_{\bar{k}}(\Omega^{\bullet}_{[0,\lambda]}(M,\mathcal{E}),\nablahsharp)$ of the complex $(\Omega^{\bullet}_{[0,\lambda]}(M,\mathcal{E}),\nablahsharp)$ is naturally isomorphic to the homology $H_{\bar{k}}(\Omega^{\bullet}(M,\mathcal{E}),\nablahsharp) \cong H^{\bar{k}}(M,\mathcal{E}',\mathcal{H})$, cf. \eqref{E:hocodurel1p}, and the cohomology $H^{\bar{k}}(\Omega^{\bullet}_{[0,\lambda]}(M,\mathcal{E}),\nabla^\mathcal{H})$ of the complex $(\Omega^{\bullet}_{[0,\lambda]}(M,\mathcal{E}),\nabla^\mathcal{H})$ is naturally isomorphic to the cohomology $H^{\bar{k}}(M,\mathcal{E},\mathcal{H})$.

Similar to the $\mathbb{Z}$-graded case, cf. \cite[Section 8]{CM}, the chirality operator $\Gamma$ establishes a complex linear isomorphism of the homology groups with cohomology groups
\[
H_{\bar{k}}(\Omega^{\bullet}_{[0,\lambda]}(M,\mathcal{E}),\nablahsharp) \cong H^{\overline{m-k}}(\Omega^{\bullet}_{[0,\lambda]}(M,\mathcal{E}),\nabla^\mathcal{H}) \cong H^{\overline{m-k}}(M,\mathcal{E},\mathcal{H}). 
\]
In particular, we have the following isomorphism
\begin{equation}\label{E:detis}
\Det H_\bullet(\Omega^{\bullet}(M,\mathcal{E}),\nablahsharp) \cong \Det H_{\bullet}(\Omega^{\bullet}_{[0,\lambda]}(M,\mathcal{E}),\nablahsharp) \cong \big(\Det H^\bullet(M,\mathcal{E},\mathcal{H})\big)^{(-1)^m}.
\end{equation}
Using the Poincar\'{e} duality, we also have the isomorphism 
\begin{equation}\label{E:poin}
\Det H^{\overline{m-k}}(M,\mathcal{E},\mathcal{H})^{-1} \cong \Det H^{\bar{k}}(M,\mathcal{E}',\mathcal{H}), 
\end{equation}
where $\mathcal{E}'$ is the dual vetor bundle of the vector bundle $E$. Therefore, we have
\begin{align}\label{E:isomo}
& \Det H^\bullet(M,\mathcal{E},\mathcal{H}) \otimes \Det H^{\overline{m-\bullet}}(M,\mathcal{E},\mathcal{H})^{-1} \nonumber \\
\cong &  \Det H^\bullet(M,\mathcal{E},\mathcal{H}) \otimes \Det H^\bullet(M,\mathcal{E}',\mathcal{H}) \quad \text{by} \, \eqref{E:poin} \\
\cong & \Det H^\bullet(M,\mathcal{E}\oplus \mathcal{E}',\mathcal{H}). \nonumber
\end{align}

For each $k=0,1$, set
\begin{equation}\label{E:omegaspli11}
\begin{array}{l}
\Omega^{\bar k}_{+,\mathcal{I}}(M,\mathcal{E}):= \Ker(\nabla^\mathcal{H} \nabla^{\mathcal{H},\sharp}) \cap \Omega^{\bar k}_{\mathcal{I}}(M,\mathcal{E}),\\
\\
\Omega^{\bar k}_{-,\mathcal{I}}(M,\mathcal{E}):= \Ker(\nabla^{\mathcal{H},\sharp}\nabla^\mathcal{H}) \cap \Omega^{\bar k}_{\mathcal{I}}(M,\mathcal{E}).
\end{array}
\end{equation}
Clearly,
\[
\Omega^{\bar k}_{\mathcal{I}}(M,\mathcal{E}) = \Omega^{\bar k}_{+,\mathcal{I}}(M,\mathcal{E}) \oplus \Omega^{\bar k}_{-,\mathcal{I}}(M,\mathcal{E}), \ \text{if} \ 0 \notin \mathcal{I}.
\]

Let $\theta \in (0,2\pi)$ be an Agmon angle, cf. \cite{Sh}. Since the leading symbol of $\nabla^{\mathcal{H},\sharp}\nabla^\mathcal{H}$ is positive definite, the $\zeta$-regularized determinant $\Det_\theta (\nabla^{\mathcal{H},\sharp}\nabla^\mathcal{H})\big|_{\Omega^{\bar k}_{+,\mathcal{I}}(M,\mathcal{E})}$ is independent of the choice of $\theta$.

For any $0 \le \lambda \le \mu \le \infty$, one easily sees that 
\begin{align*}
\prod_{k=0,1} \big(\,\Det_\theta (\nabla^{\mathcal{H},\sharp}\nabla^\mathcal{H})\big|_{\Omega^{\bar k}_{+,(\lambda,\infty)}(M,\mathcal{E})} \, \big)^{(-1)^k}  = & \Big[\prod_{k=0,1} \big(\,\Det_\theta (\nabla^{\mathcal{H},\sharp}\nabla^\mathcal{H})\big|_{\Omega^{\bar k}_{+,(\lambda,\mu)}(M,\mathcal{E})} \, \big)^{(-1)^k} \Big] \\
& \cdot \Big[\prod_{k=0,1} \big(\,\Det_\theta (\nabla^{\mathcal{H},\sharp}\nabla^\mathcal{H})\big|_{\Omega^{\bar k}_{+,(\mu,\infty)}(M,\mathcal{E})} \, \big)^{(-1)^k} \Big]  
\end{align*}

For any $\lambda \ge 0$, denote by $\tau_{[0,\lambda]}$ the Cappell-Miller torsion of the twisted de Rham bi-graded complex $\big(\Omega^\bullet_{[0,\lambda]}(M,\mathcal{E}),\nabla^\mathcal{H}, \nablahsharp \big)$. Via the isomorphisms $H_{\bullet}(\Omega^{\bullet}_{[0,\lambda]}(M,\mathcal{E}),\nablahsharp) \cong H_\bullet(\Omega^{\bullet}(M,\mathcal{E}),\nablahsharp)$ and $H^\bullet(\Omega^\bullet_{[0,\lambda]}(M,\mathcal{E}),\nabla^\mathcal{H}) \cong H^\bullet(M,\mathcal{E},\mathcal{H})$ and \eqref{E:isomo}, we can view $\tau_{[0,\lambda]}$ as an element of
$\Det H^\bullet(M,\mathcal{E} \oplus \mathcal{E}',\mathcal{H})$. 
In particular, if $m$ is odd, then, up to an isomorphism,
\begin{equation}\label{E:tauoddre}
\tau_{[0,\lambda]} \in  \Det H^\bullet(M,\mathcal{E},\mathcal{H}) \otimes \Det H^\bullet(M,\mathcal{E},\mathcal{H}) \cong \Det H^\bullet(M, \mathcal{E} \oplus \mathcal{E}', \mathcal{H}).
\end{equation}

The proof of the following lemma is similar to the proof of \cite[Theorem 8.3]{CM}, we omit the proof.
\begin{lemma}
The element 
\[
\tau_{[0,\lambda]} \cdot \prod_{k=0,1}{ \big(\, \Det_\theta (\nabla^{\mathcal{H},\sharp}\nabla^\mathcal{H})\big|_{\Omega^{\bar k}_{+,(\lambda,\infty)}(M,\mathcal{E})}\,\big)^{(-k)}}
\]
is independent of the choice of $\lambda$.
\end{lemma}

We now define the Cappell-Miller analytic torsion for the de Rham complex twisted by a flux.
\begin{definition}\label{D:cmantorsi}
Let $(\mathcal{E},\nabla)$ be a complex vector bundle over a connected oriented $m$-dimensional closed Riemannian manifold $M$ and $\mathcal{H}$ is a closed odd degree form, other than one form. Further, let $\nabla^\mathcal{H} = \nabla + \mathcal{H} \wedge \cdot$ and $\nablahsharp = \Gamma \nabla^\mathcal{H} \Gamma$. Let $\theta \in (0,2\pi)$ be an Agmon angle for the operator $\Delta^{\mathcal{H},\sharp}:=(\nabla^\mathcal{H}+\nablahsharp)^2$. The Cappell-Miller torsion $\tau(\nabla,\mathcal{H})$ for the twisted de Rham bi-graded complex $(\Omega^\bullet(M,\mathcal{E}),\nabla^\mathcal{H}, \nablahsharp)$ is an element of $\Det H^\bullet(M,\mathcal{E},\mathcal{H})  \otimes \big( \Det  H^\bullet(M,\mathcal{E},\mathcal{H}) \big)^{(-1)^{m+1}}$ defined as follows
\begin{equation}\label{E:cmttordef}
\tau(\nabla,\mathcal{H}) \, := \, \tau_{[0,\lambda]} \cdot  \prod_{k=0,1}{ \big(\, \Det_\theta (\nabla^{\mathcal{H},\sharp}\nabla^\mathcal{H})\big|_{\Omega^{\bar k}_{+,(\lambda,\infty)}(M,\mathcal{E})}\,\big)^{(-k)}}.
\end{equation} 
\end{definition}

\subsection{Twisted Cappell-Miller analytic torsion under metric and flux deformations}
In this subsection we obtain the variation formulas for the twisted Cappell-Miller analytic torsion $\tau(\nabla,\mathcal{H})$ under the metric and flux deformations. In particular, we show that if the manifold $M$ is an odd dimensional closed oriented manifold, then the twisted Cappell-Miller analytic torsion is independent of the Riemannian metric and the representative $\mathcal{H}$ in the cohomology class $[\mathcal{H}]$. See also \cite{Su}.

The proof of the following theorem is similar to the proof of Theorem \ref{T:vartheorem1}. We omit the proof.
\begin{theorem}\label{T:anvartheorem1}
Let $(\mathcal{E},\nabla)$ be a complex vector bundle over a $m$-dimensional connected oriented closed Riemannian manifold $M$ and $\mathcal{H}$ is a closed odd degree form, other than one form. Let $g^M_v, a \le v \le b$, be a smooth family of Riemannian metrics on $M$, then the corresponding twisted Cappell-Miller analytic torsion $\tau_v(\nabla,\mathcal{H})$ varies smoothly and the variation of $\tau_v(\nabla,\mathcal{H})$ is given by a local formula
\[
\frac{d}{dv} \tau_v(\nabla,\mathcal{H}) = \big( \sum_{k=0,1}(-1)^k \int_M b_{m/2,\bar{k},v} \big) \cdot \tau_v(\nabla,\mathcal{H}).
\]
In particular, if the dimension of the manifold $M$ is odd, then twisted Cappell-Miller analytic torsion $\tau(\nabla,\mathcal{H})$ is independent of the Riemannian metric $g^M$.
\end{theorem}

For the untwisted case considered in \cite{BZ}, the variation of the torsion can be integrated to an anomaly formula \cite{BZ}.

The proof of the follwoing is similar to Theorem 6.1 of \cite{MW3}. See also Theorem 3.8 of \cite{MW}.
\begin{theorem}\label{T:anvartheorem2}
Let $(\mathcal{E},\nabla)$ be a complex vector bundle over a $m$-dimensional connected oriented closed Riemannian manifold $M$ and $\mathcal{H}$ is a closed odd degree form, other than one form. Under the natural identification of determinant lines and along any one parameter deformation $\mathcal{H}_v$ of $\mathcal{H}$ that fixes the cohomology class $[\mathcal{H}]$, we have the following variation formula
\[
\frac{d}{dv} \tau(\nabla,\mathcal{H}_v) = \big( \sum_{k=0,1}(-1)^k \int_M c_{m/2,\bar{k},v} \big) \cdot \tau(\nabla,\mathcal{H}_v).
\] 
In particular, if the dimension of the manifold $M$ is odd, then, under the natural identification of determinant lines, the twisted Cappell-Miller analytic torsion $\tau(\nabla,\mathcal{H})$ is independent of any deformation of $\mathcal{H}$ that fixes the cohomology class $[\mathcal{H}]$.
\end{theorem}

\subsection{Relationship with the twisted refined analytic torsion}
In this subsection we assume that $M$ is a closed compact oriented manifold of odd dimension. Recall that in \cite[(3.13)]{H1}, for each $\lambda > 0$, we define the twisted refined torsion $\rho_{\Gamma_{[0,\lambda]}}$ of the twisted finite dimensional complex $(\Omega^\bullet_{[0,\lambda]}(M,\mathcal{E}),\nabla^\mathcal{H})$ corresponding to the chirality operator $\Gamma_{[0,\lambda]}$. In our setting, as in the $\mathbb{Z}$-graded case, \cite[(5.1)]{BK6}, the twisted Cappell-Miller torsion can be described as, \eqref{E:tauoddre},
\begin{equation}\label{E:refinedcmtorrel1}
\tau_{[0,\lambda]}:= \rho_{\Gamma_{[0,\lambda]}} \otimes \rho_{\Gamma_{[0,\lambda]}} \in \Det H^\bullet(M,\mathcal{E},\mathcal{H}) \otimes \Det H^\bullet(M,\mathcal{E},\mathcal{H}).
\end{equation} 
 
By combining (3.14), (3.20), (5.28) and Definition 4.5 in \cite{H1}, the twisted refined analytic torsion can be written as
\begin{equation}\label{E:refinedcmtorrel2}
\rho_{\operatorname{an}}(\nabla^\mathcal{H})= \pm \rho_{\Gamma_{[0,\lambda]}} \cdot \prod_{k=0,1}{ \big(\, \Det_\theta (\nabla^{\mathcal{H},\sharp}\nabla^\mathcal{H})\big|_{\Omega^{\bar k}_{+,(\lambda,\infty)}(M,\mathcal{E})}\,\big)^{-\frac{k}{2}}} \cdot \exp \big( -i \pi (\eta(\mathcal{B}^\mathcal{H}_{\bar 0}(\nabla^\mathcal{H}))- \operatorname{rank}\mathcal{E} \cdot \eta_{\operatorname{trivial}})  \big),
\end{equation}  
where $\eta(\mathcal{B}^\mathcal{H}_{\bar 0}(\nabla^\mathcal{H}))- \operatorname{rank}\mathcal{E} \cdot \eta_{\operatorname{trivial}}$ is the $\rho$-invariant of the twisted odd signature operator $\mathcal{B}^\mathcal{H}_{\bar 0}(\nabla^\mathcal{H})$, defined in \cite[(3.2)]{H1}.

By combining \eqref{E:cmttordef}, \eqref{E:refinedcmtorrel1} with \eqref{E:refinedcmtorrel2}, we have the following comparison theorem of the twisted Cappell-Miller analytic torsion and twisted refined analytic torsion.
\begin{theorem}\label{T:cmrefantorcomp}
Let $(\mathcal{E},\nabla)$ be a complex vector bundle over a connected oriented odd-dimensional closed Riemannian manifold $M$ and $\mathcal{H}$ is a closed odd degree form, other than one form. Further, let $\nabla^\mathcal{H} = \nabla + \mathcal{H} \wedge \cdot$. Then
\[
\tau(\nabla,\mathcal{H}) \cdot \exp \big( -2i \pi (\eta(\mathcal{B}^\mathcal{H}_{\bar 0}(\nabla^\mathcal{H}))- \operatorname{rank}\mathcal{E} \cdot \eta_{\operatorname{trivial}})  \big) = \rho_{\operatorname{an}}(\nabla^\mathcal{H}) \otimes \rho_{\operatorname{an}}(\nabla^\mathcal{H}).
\]
\end{theorem}

Note that in Theorem 5.1 of \cite{Su}, Su compared the twisted Burghelea-Haller analytic torsion introduced in \cite{Su} with the twisted refined analytic torsion. By combining Theorem 5.1 of \cite{Su} with Theorem \ref{T:cmrefantorcomp}, we can also obtain the comparison theorem of the twisted Burghelea-Haller torsion and the twisted Cappell-Miller analytic torsion. We skip the details.

%%%%%%%%%%%%%%%%%%%%%%%%%%%%%%%%%%%%%%%%%%%%%%%%%%%%%%%%%%%%%%%%%%%%%%%%%%%%%%%%%%%%%%%%%%%%

\end{document}